\documentclass[11pt]{article}
\usepackage{amsmath}
\usepackage{amsthm}
\usepackage{amssymb}
\usepackage{amscd}
\usepackage{relsize}
\usepackage{enumerate}
\usepackage{centernot}
\usepackage{mathtools}
\usepackage{ stmaryrd }
\newtheorem{dfn}{Definition}
\newtheorem{teo}{Theorem}

\newtheorem{cor}[teo]{Corollary}
\newtheorem{lema}[teo]{Lemma}
\newcommand{\CC}{\mathbb{C}}
\newcommand{\RR}{\mathbb{R}}
\newcommand{\ZZ}{\mathbb{Z}}

\newcommand{\Hc}{\mathcal{H}}
\newcommand{\Tc}{\mathcal{T}}

\newcommand{\Sc}{\mathcal{S¼}}
\newcommand{\Mc}{\mathcal{M}}
\newcommand{\Ac}{\mathcal{A}}
\newcommand{\Lc}{\mathcal{L}}
\newcommand{\espan}{\operatorname{span}}

\DeclareMathOperator*{\esup}{ess\,sup}
\DeclareMathOperator*{\einf}{ess\,inf}

\begin{document}

\title{\bf Riesz bases associated with regular representations of semidirect product groups}
\author{
{\bf Antonio G. Garc\'{\i}a}\thanks{E-mail:\texttt{agarcia@math.uc3m.es}}
{\bf\,\, and \, Gerardo P\'erez-Villal\'on}\thanks{E-mail:\texttt{gperez@euitt.upm.es}}
}
\date{}
\maketitle
\begin{itemize}
\item[*] Departamento de Matem\'aticas, Universidad Carlos III de Madrid,
 Avda. de la Universidad 30, 28911 Legan\'es-Madrid, Spain.
 \item[\dag] Departamento de Matem\'aticas Aplicada a las Tecnolog\'{\i}as de la Informaci\'on y las Comunicaciones, E.T.S.I.S.T., Universidad Polit\'ecnica de Madrid,
 Carret. Valencia km.~7, 28030 Madrid, Spain.
\end{itemize}
\begin{abstract}
This work is devoted to the study of  Bessel and Riesz systems of the type  $\big\{L_{\gamma}\mathsf{f}\big\}_{\gamma\in \Gamma}$ obtained from the action of the left regular representation $L_{\gamma}$ of a discrete non abelian group $\Gamma$ which is  a semidirect product, on a function $\mathsf{f}\in \ell^2(\Gamma)$. The main features about these systems can be conveniently studied  by means of a simple matrix-valued function $\mathbf{F}(\xi)$. These  systems allow to derive  sampling results in principal $\Gamma$-invariant spaces, i.e., spaces obtained from the action of the group $\Gamma$ on a element of a Hilbert space. Since the systems $\big\{L_{\gamma}\mathsf{f}\big\}_{\gamma\in \Gamma}$ are closely related to convolution operators, a connection with $C^*$-algebras  is also established.\end{abstract}

\noindent{\bf AMS}: 42C15; 20H15; 94A20.
\section{Introduction}

This work is devoted to the study of a characterization as Riesz bases, together with some sampling applications, of systems  $\big\{L_{\gamma}\mathsf{f}\big\}_{\gamma\in \Gamma}$ obtained from the left regular representation of a discrete non abelian group $\Gamma$, that is, $L_{\gamma}\mathsf{f}(\eta):= \mathsf{f}(\gamma^{-1}\eta)$, $\eta, \gamma\in \Gamma$, where $\mathsf{f}$ denotes a fixed element in the  Hilbert space $\ell^2(\Gamma)$. Throughout the paper the group $\Gamma:=N\rtimes_\sigma H$ is the semidirect product of two groups:  a discrete abelian group $N$ and  a finite group $H$; the subscript $\sigma$ denotes the action of the group $H$ on the group $N$. Some important examples of non abelian groups such as dihedral groups, infinite  dihedral  group or crystallographic groups are semidirect products  with these characteristics. 

In addition to the intrinsic importance of  the left regular representation 
\,$ \gamma \mapsto L_\gamma \in \mathcal{U}\big(\ell^2(\Gamma)\big)$ in representation theory of groups,
 the systems $\big\{L_{\gamma}\mathsf{f}\big\}_{\gamma\in \Gamma}$ arising from the left regular representation of $\Gamma$ are relevant in applications; for instance they appear in sampling theory. In fact, in the present paper we deal with two types of samples where these systems have an important role: 
 
 Firstly, given a unitary representation $\Gamma \ni \gamma \mapsto U(\gamma)\in \mathcal{U}(\Hc)$ of the group $\Gamma$ on a separable Hilbert space $\Hc$, for a fixed $\varphi \in \Hc$ we consider the subspace of $\Hc$
\[
\Ac_\varphi= \Big\{ \sum_{\gamma\in \Gamma} \mathsf{a}(\gamma) U(\gamma)\varphi \,:\,\,
\mathsf{a}=\{\mathsf{a}(\gamma)\}_{\gammaÊ\in \Gamma} \in \ell^2(\Gamma)\Big\}
\]
For a fixed $\psi \in \Hc$, which does not necessarily belong  to $\Ac_\varphi$, we can define for each $f\in \Ac_\varphi$ its samples 
\[
\Lc_\psi f(\gamma):=\langle f, U(\gamma)\psi\rangle_\Hc\,, \quad \gamma\in \Gamma\,. 
\]
These samples give average sampling in classical shift-invariant subspaces of $L^2(\RR)$ (see, for instance, Refs.~\cite{aldroubi:05,hector:14,garcia:09,lei:97}). As we will see in Section \ref{section4}, there exists 
$\mathsf{f}_\psi \in \ell^2(\Gamma)$  such that, for each $f\in \Ac_\varphi$, we get $\Lc_\psi f(\gamma)=\langle \mathsf{a}, L_\gamma\mathsf{f}_\psi \rangle_{\ell^2(\Gamma)}$, $\gamma\in \Gamma$, being $\mathsf{a} \in \ell^2(\Gamma)$ the coefficients sequence of $f\in \Ac_\varphi$. 

Secondly, when $\Hc=L^2(\RR^d)$, for a fixed point $p\in \RR^d$ we consider, for any $f\in \Ac_\varphi$ the samples 
\[
\Lc_pf(\gamma):=\big[U(\gamma^{-1})f\big](p)\,, \quad \gamma \in \Gamma\,.
\]
Again, there exists $\mathsf{f}_p \in \ell^2(\Gamma)$  such that, for each $f\in \Ac_\varphi$, we obtain the expression for these samples $\Lc_p f(\gamma)=\langle \mathsf{a}, L_\gamma\mathsf{f}_p\rangle_{\ell^2(\Gamma)}$, $\gamma\in \Gamma$. This situation englobes the case when we are dealing with pointwise samples in  classical shift-invariant subspaces of $L^2(\RR)$ \cite{garcia:09,kim:08,lei:97,selvan:15,zhou:99}. 

Needless to say that a feasible characterization of the system $\big\{L_{\gamma}\mathsf{f}\big\}_{\gamma\in \Gamma}$ as a Riesz basis for $\ell^2(\Gamma)$ along with the search of its dual Riesz basis will play a crucial role in obtaining an interpolatory  sampling formula allowing the recovery of any $f\in \Ac_\varphi$ from the given data sampling $\{\Lc_\psi f(\gamma)\}_{\gamma \in \Gamma}$ or $\{\Lc_pf(\gamma)\}_{\gamma \in \Gamma}$. Some $\Gamma$-invariant spaces $\mathcal{A}_{\varphi}$ of special relevance  are those appearing in composite wavelet theory. These wavelets allow many more locations, scales and directions than the classical ones. They have been studied in the last few years; see, for instance, Refs.~\cite{gonzalez:11,guo:06,mac:11,manning:12}.

\medskip

In order to get a suitable characterization of when the system $\big\{L_{\gamma}\mathsf{f}\big\}_{\gamma\in \Gamma}$ is a Riesz basis for $\ell^2(\Gamma)$, we briefly describe the mathematical techniques used in this paper.
Since, for $\mathsf{a}, \mathsf{f}\in \ell^2(\Gamma)$, we have
\[
\sum_{\eta\in \Gamma} \mathsf{a}(\eta)L_{\eta}\mathsf{f}=\mathsf{a}\ast \mathsf{f} \quad \text{and }\quad 
\langle \mathsf{a}, L_{\gamma}\mathsf{f} \rangle_{\ell^2(\Gamma)}=\big( \mathsf{a} \ast \mathsf{f}^*\big)(\gamma)\,,
\]
where $\mathsf{f}^*$ denotes the involution in $\ell^2(\Gamma)$ of $\mathsf{f}$, we can use techniques of linear time-invariant (LTI) systems. In fact, the given characterization will be described in terms of a matrix-valued function $\mathbf{F}(\xi)$ which turns out to be the transfer matrix of a multi-input multi-output (MIMO) system. Although, our group $\Gamma=N\rtimes_\sigma H$ is not abelian and a classical Fourier analysis  is not directly applicable, the MIMO system formalism will allow us to make use of the Fourier transform on the locally compact abelian (LCA) group $N$; in fact, $\mathbf{F}(\xi)$ is defined for $\xi$ in $\widehat{N}$, the dual group of characters $\xi$ of  $N$.

\medskip

As the above equalities show, the study of the systems $\big\{L_{\gamma}\mathsf{f}\big\}_{\gamma\in \Gamma}$ is related to convolution algebras on $\Gamma$, and specially,  to the representation of the group $C^*$-algebra of $\Gamma$, $C^*(\Gamma)$, given in Ref.~\cite{taylor:89}. In the last section we show this relationship and, in so doing, we provide a convolution $C^*$-algebra larger than $C^*(\Gamma)$ and suitable for the present context. 

\medskip

The paper is organized as follows: Section 2 provides the mathematical setting needed throughout the paper giving the keys for different approaches, together with some lemmata used in the sequel. Section \ref{section3} includes the main theoretical results in the paper: A characterization of  when  $\big\{L_{\gamma}\mathsf{f}\big\}_{\gamma\in \Gamma}$ is a Bessel sequence or a Riesz basis  for $\ell^2(\Gamma)$ is respectively proved in Theorems \ref{bessel} and \ref{riesz}; in particular the orthonormal basis case is considered (Corollary \ref{ort}). The dual  Riesz basis of $\big\{L_{\gamma}\mathsf{f}\big\}_{\gamma\in \Gamma}$, which has the same form $\big\{L_{\gamma}\mathsf{g}\big\}_{\gamma\in \Gamma}$, is also obtained (Corollary \ref{dual}). Closing the section, an example of Riesz bases associated to the infinite dihedral group $D_\infty$ is proposed.  Section \ref{section4} is devoted to a sampling application of the results in Section \ref{section3}; thus, an abstract sampling result  is obtained (Theorem \ref{sampriesz}) in a principal $U$-invariant subspace of a Hilbert space $\Hc$. An example using crystallographic groups illustrates the sampling results, where we consider pointwise samples as well as average samples. Finally, in Section \ref{section5} a $C^*$-algebra connection is also exhibited.

\section{The mathematical setting}
\label{section2}
The aim of this section is twofold: Firstly, to provide a brief on the needed mathematical preliminaries and, secondly, to establish the mathematical setting which will give us the keys and tools to get the main aim in the paper, i.e., a suitable characterization of the system 
$\big\{L_{\gamma}\mathsf{f}\big\}_{\gamma\in \Gamma}$ as a Riesz basis for $\ell^2(\Gamma)$.

\medskip

We begin stating the main facts concerning the group $\Gamma$ as a semidirect product of groups. Let $(N,+)$ be a discrete abelian group,  $(H,\cdot)$ a finite group of order $\kappa$, and a homomorphism $\sigma: H \mapsto \text{Aut}(N)$ referred as the action of the group $H$ on the group $N$. Its {\em semidirect product} 
$\Gamma:=N\rtimes_\sigma H$ is the group whose elements are the pairs $(n,h)\in N\times H$ with multiplication rule 
\[
(n,h)(m,l):=\big( n\, +\, \sigma(h)(m), h\, l \big)\,, \quad n,m\in N\,\,\text{and}\,\, h,l\in H\,. 
\]
In particular, the identity element in $\Gamma$ is $e_\Gamma=(0_N,1_H)$ and $(n,h)^{-1}=\big(-\sigma(h^{-1})(n),h^{-1}\big)$, \,$(n,h)\in \Gamma$. Notice that, unless $\sigma(h)$ equals the identity for $h\in H$, the group $\Gamma$ is not abelian.

\medskip

An important example of semidirect product of groups that will be used in this paper is the {\em crystallographic group} $\Gamma_{\Mc,H}=\Mc \ZZ^d \rtimes_{\sigma}H$, where $\Mc$ is a  non-singular $d\times d$ matrix and $H$ is a finite subgroup of $O(d)$, the orthogonal group of order $d$, such that $A(\Mc\ZZ^d) = \Mc\ZZ^d$ for all $A\in H$. Here $\sigma(A)x=Ax$ for $A\in H$ and $x\in \RR^d$. The {\em infinite dihedral group} $D_\infty:=\ZZ \rtimes_\sigma \ZZ_2$, where $\sigma(1)(n)=n$ and $\sigma(-1)(n)=-n$ for each $n\in \ZZ$, is a unidimensional crystallographic group.

\medskip

Throughout the paper we denote by greek letters $\gamma, \eta,\dots$ or as $(n,h), (m,l),\ldots$ the elements in $\Gamma$. The {\em left regular representation} of the group $\Gamma$ on $\ell^2(\Gamma)$ is given by 
\[
L_{\gamma}\mathsf{f}(\eta):= \mathsf{f}(\gamma^{-1}\eta),\quad \eta, \gamma\in \Gamma\,\text{ andÊ}\,\, \mathsf{f}\in \ell^2(\Gamma).
\]
Note that, for each $\mathsf{f}\in \ell^2(\Gamma)$, the  {\em synthesis operator} $\Lambda_{\mathsf{f}}$ of the system $\big\{L_{\gamma}\mathsf{f}\big\}_{\gamma\in \Gamma}$  is a convolution operator. Indeed, for any $\mathsf{a}\in \ell^2(\Gamma)$
 \[
\Lambda_{\mathsf{f}}\mathsf{a}(\gamma):=\sum_{\eta\in \Gamma} \mathsf{a}(\eta)L_{\eta}\mathsf{f}(\gamma)= \sum_{\eta\in \Gamma} \mathsf{a}(\eta)\mathsf{f}(\eta^{-1}\gamma)=\big(\mathsf{a}\ast \mathsf{f}\big)(\gamma)\,,\,\, \gamma \in \Gamma\,.
 \]
 Recall that  the above definition gives a bounded linear operator $\Lambda_{\mathsf{f}}:\ell^2(\Gamma)\mapsto \ell^2(\Gamma)$ if and only if the system $\big\{L_{\gamma}\mathsf{f}\big\}_{\gamma\in \Gamma}$ is a Bessel sequence \cite[Theorem 3.2.3]{ole:16} (see also Theorem \ref{bessel} infra).
Our study is based on the following representation of  $\Lambda_{\mathsf{f}}\mathsf{a}$:
\begin{lema}
Given $\mathsf{f}, \mathsf{a} \in \ell^2(\Gamma)$, $\Lambda_{\mathsf{f}}\mathsf{a}$ is represented as:
 \begin{equation}
 \label{sintesis}
 \Lambda_{\mathsf{f}}\mathsf{a}(n,h)= \sum_{l\in H} \big(\mathsf{a}_{l} \ast_{_N} \mathsf{f}_{h,l}\big)(n)\,,\quad (n,h)\in \Gamma\,,
\end{equation}
where $\mathsf{a}_{l}(n):=\mathsf{a}(n,l)$, $\mathsf{f}_{h,l}(n):=\mathsf{f}[(0,l)^{-1}(n,h)]$,
and $\ast_{_N}$ denotes the convolution on the abelian group $N$, i.e., $\big(a\ast_{_N} b\big)(n):=\sum_{m\in N}a(m)\,b(n-m)$.  
\end{lema}
\begin{proof}
Having in mind that $(m,l)^{-1}=(0,l)^{-1}(-m,1)$, we get 
\[
\Lambda_{\mathsf{f}}\mathsf{a}(n,h)=\sum_{l\in H} \sum_{m\in N}\mathsf{a}(m,l)\,\mathsf{f}[(m,l)^{-1}(n,h)]=\sum_{l\in H} \sum_{m\in N}\mathsf{a}(m,l)\,\mathsf{f}[(0,l)^{-1}
  (n-m,h)],
  \]
and representation \eqref{sintesis} holds.
\end{proof}
According to expression \eqref{sintesis}, the  operator $\Lambda_{\mathsf{f}}$ can be seen as
a linear time-invariant system:
\begin{equation}
\label{mimo}
[\mathsf{a}_{l}]_{l\in H}\in \ell^2_\kappa(N) \longmapsto 
\Big[\sum_{l\in H} \mathsf{a}_{l} \ast_{_{N}} \mathsf{f}_{h,l}\Big]_{h\in H}
\end{equation}
where $\ell^2_\kappa(N):=\ell^2(N)\times\dots \times \ell^2(N)$ ($\kappa$ times).
In signal processing jargon this type of system is called  a multi-input, multi-output (MIMO) linear time-invariant system; it can be effectively analyzed by using the Fourier transform. 

\medskip

Since $N$ is a discrete abelian group, we can use the Fourier transform on $N$ defined  by   
\[\widehat{\mathsf{a}}(\xi)=\sum_{n\in N} \mathsf{a}(n) \langle -n,\xi \rangle,\quad \xi \in \widehat{N},\]  
for any $\mathsf{a}\in \ell^1(N)$, and extended to $\ell^2(N)$ as a unitary operator between $\ell^2(N)$ and $L^2(\widehat{N})$ where $\widehat{N}$ denotes the dual group of characters (see, for instance, Ref.\cite{folland:95} for the details).

Next lemma will be needed in taking the Fourier transform in Eq.~\eqref{sintesis}. It also gives a condition so that the output in \eqref{mimo}  belongs to 
$\ell^2_\kappa(N)$.
 \begin{lema}
 \label{conv}
Let $\mathsf{a},\mathsf{b} \in \ell^2(N)$ such that the product \,$\widehat{\mathsf{a}}(\xi)\,\widehat{\mathsf{b}}(\xi)\in L^2(\widehat{N})$. Then the convolution 
$\mathsf{a}\ast_{_{N}} \mathsf{b}\in \ell^2(N)$  and
 \[
\big( \widehat{\mathsf{a}\ast_{_{N}} \mathsf{b}}\big)(\xi)=\widehat{\mathsf{a}}(\xi)\,\widehat{\mathsf{b}}(\xi),\quad \text{a.e.}\,\, \xi\in \widehat{N}\,.
 \] 
 \end{lema}
 \begin{proof}
 By using Plancherel theorem \cite[Theorem 4.25]{folland:95} and denoting $\widetilde{\mathsf{b}}(n)=\overline{\mathsf{b}(-n)}$, we obtain
 \[
 \begin{split}
 \big(\mathsf{a}\ast_{_{N}} \mathsf{b}\big)(n)&=\sum_{m\in N}\mathsf{a}(m)\,\mathsf{b}(n-m)=\big\langle \mathsf{\mathsf{a}},\widetilde{\mathsf{b}}(\cdot-n)\big\rangle_{\ell^2(N)}=\big\langle \widehat{\mathsf{a}},\widehat{\widetilde{\mathsf{b}}(\cdot-n)}\big\rangle_{\ell^2(\widehat{N})}\\ 
 &=\int_{\widehat{N}}  \widehat{\mathsf{a}}(\xi)\,\overline{\widehat{\widetilde{\mathsf{b}}}(\xi)} \, \overline{\langle -n,\xi \rangle}d\mu_{\widehat{N}}(\xi)=\int_{\widehat{N}}  \widehat{\mathsf{a}}(\xi)\,\widehat{\mathsf{b}}(\xi) \, \overline{\langle -n,\xi \rangle}d\mu_{\widehat{N}}(\xi)\,.
 \end{split}
 \]
 Since $\big\{\langle-n,\xi \rangle\big\}_{n\in N}$ is an orthonormal basis for $L^2(\widehat{N})$ 
\cite[Theorem 4.26]{folland:95} and we have assumed that $\widehat{\mathsf{a}}(\xi)\,\widehat{\mathsf{b}}(\xi)\in L^2(\widehat{N})$, we obtain that $\mathsf{a}\ast_{_{N}} \mathsf{b}\in \ell^2(N)$ and
\[
\widehat{\mathsf{a}}(\xi)\,\widehat{\mathsf{b}}(\xi)=\sum_{n\in N} \big(\mathsf{a}\ast_{_{N}} \mathsf{b}\big)(n)\langle -n,\xi \rangle,\quad \quad \text{a.e.}\,\, \xi\in \widehat{N}\,.
\]
\end{proof}
By taking the $N$-Fourier transform in the second term of expression \eqref{sintesis} we obtain the so called {\em transfer matrix}  of the MIMO system \eqref{mimo}. This motivates the following definition:
\begin{dfn}
\label{defF}
For each $\mathsf{f} \in \ell^2(\Gamma)$ we introduce its associated transfer matrix as the $\kappa\times \kappa$ matrix-valued function $\mathbf{F}$ defined on 
$\widehat{N}$ as 
\begin{equation}
\label{F}
\mathbf{F}(\xi)=\big[\,\widehat{\mathsf{f}}_{h,l}(\xi)\,\big]_{h,l\in H}\quad \text{ where }\,\, \mathsf{f}_{h,l}(n)=\mathsf{f}[(0,l)^{-1}(n,h)]\,, \,\,n\in N\,.
\end{equation}
\end{dfn}

\medskip

The {\em involution} in $\ell^2(\Gamma)$ and in $\ell^2(N)$ are denoted, respectively, by  
\[
\mathsf{f}^*(\gamma)=\overline{\mathsf{f}(\gamma^{-1}})\,, \,\, \gamma \in \Gamma \quad \text{and}\quad \widetilde{\mathsf{f}}_{h,l}(n)=\overline{\mathsf{f}_{h,l}(-n)}\,, \,\, n\in N\,.
\]
The role of the conjugate transpose matrix-valued function $\mathbf{F}^*(\xi)$ is also well understood realizing that it
is the transfer matrix  of the  system 
\[
[\mathsf{a}_{l}]_{l\in H} \longmapsto \Big[\sum_{l\in H} \mathsf{a}_{l} \ast_{_{N}} \widetilde{\mathsf{f}}_{l,h}\Big]_{h\in H}
\]
which represents the {\em analysis operator}  of the sequence $\big\{L_{\gamma}\mathsf{f}\big\}_{\gamma\in \Gamma}$. 
Namely, for $\mathsf{f}, \mathsf{a}$ in $\ell^2(\Gamma)$, we have
\begin{equation}
\label{analisis}
\Ac_{\mathsf{f}}(n,h):=\big\langle \mathsf{a}, L_{(n,h)}\mathsf{f} \big\rangle_{\ell^2(\Gamma)}= \big(\mathsf{a} \ast \mathsf{f}^*\big)(n,h)=
 \sum_{l\in H } \big(\mathsf{a}_{l}\ast_{_{N}} \widetilde{\mathsf{f}}_{l,h}\big)(n)\,.
 \end{equation}
Indeed, equality
\[
\sum_{l\in H } \sum_{m\in N}\mathsf{a}(m,l)\overline{\mathsf{f}[(n,h)^{-1}(m,l)]}
 = \sum_{l\in H } \sum_{m\in N}\mathsf{a}(m,l)\overline{\mathsf{f}[(0,h)^{-1}(m-n,l)]}
\]
yields \eqref{analisis}. Recall that the analysis operator is the adjoint operator of the synthesis operator  \cite[Lemma 3.2.1]{ole:16}; in other words, whenever the system 
$\{L_\gamma \mathsf{f}\}_{\gamma \in \Gamma}$ is a Bessel sequence for $\ell^2(\Gamma)$, operator $\Ac_{\mathsf{f}}$ in \eqref{analisis} is the adjoint operator 
of $\Lambda_{\mathsf{f}}$, i.e., $\Ac_{\mathsf{f}}=\Lambda_{\mathsf{f}}^*$.

In our context, we will need the transform $\mathcal{T}_\Gamma$ given in the following lemma, where $L^2_\kappa(\widehat{N})$ denotes the product Hilbert space $L^2(\widehat{N})\times \dots \times L^2(\widehat{N})$ ($\kappa$ times):

\begin{lema}
\label{T} 
The linear map $\mathcal{T}_\Gamma:\ell^2(\Gamma)\rightarrow L^2_\kappa(\widehat{N})$ defined by 
$\mathcal{T}_\Gamma \mathsf{a}:=[\,\widehat{\mathsf{a}}_{h}\,]_{h\in H}$, where $\mathsf{a}_{h}(n)=\mathsf{a}(n,h)$, $(n,h)\in \Gamma$,  is a unitary operator.
\end{lema}
\begin{proof}
The map $\mathcal{T}_\Gamma$  is surjective since the $N$-Fourier transform is a unitary operator between $\ell^2(N)$ and $L^2(\widehat{N})$. It is also an isometry since, for each $\mathsf{a},\mathsf{b}\in \ell^2(\Gamma)$, we have 
\[
\qquad \langle \mathsf{a},\mathsf{b}\rangle_{\ell^2(\Gamma)}=\sum_{h\in H}\langle \mathsf{a}_h,\mathsf{b}_h\rangle_{\ell^2(N)}=\sum_{h\in H}\langle \, \widehat{\mathsf{a}}_{h},\widehat{\mathsf{b}}_{h}\, \rangle_{L^2(\widehat{N})}=\big\langle \,\mathcal{T}_\Gamma\mathsf{a}\, , \, \mathcal{T}_\Gamma\mathsf{b}\, \big\rangle_{L^2_\kappa(\widehat{N})}\,.
\]
\end{proof}
The above lemma, related to the abstract version of the Zak transform (see, for instance, Refs.\cite{barbieri:15b, hernandez:10}), says that any $\mathsf{a}\in \ell^2(\Gamma)$ is completely determined by the Fourier transform of its $\kappa$ {\em phases} $\mathsf{a}_{h}$, $h\in H$.

\medskip
 
With a view to built the matrix-valued function $\mathbf{F}(\xi)$ and the vector-valued function $\mathcal{T}_{\Gamma}\mathsf{a}(\xi)$, indexed by the elements of $H$, we order the $\kappa$ elements of $H$ such that the first element is $1_{H}$, the identity element of $H$. Thus the first column of $\mathbf{F}(\xi)$ is $\mathcal{T}_\Gamma\mathsf{f}(\xi)$. Notice also that its $l$-column is $\mathcal{T}_\Gamma L_{(0_N,l)}\mathsf{f}(\xi)$. Hence, for $\xi \in \widehat{N}$ matrix $\mathbf{F}(\xi)$ is a redundant matrix, similar to what happens with the modulation matrix in wavelet or in filter bank theory. 

\medskip
 
In next result, we obtain a representation for the analysis and synthesis operators, Eqs. \eqref{sintesis} and \eqref{analisis} respectively, in the $\mathcal{T}_\Gamma$ domain:
 
\begin{teo}
\label{Tdomain} 
Assume that  $\mathsf{f},\mathsf{a} \in \ell^2(N)$ and that  the products $\widehat{\mathsf{a}}_{l}(\xi)\,\widehat{\mathsf{f}}_{h,h'}(\xi)\in L^2(\widehat{N})$\, for all $l,h,h'\in H$. Then
\[
\mathcal{T}_\Gamma \Lambda_{\mathsf{f}}\mathsf{a}(\xi) = \mathbf{F}(\xi) \, \mathcal{T}_\Gamma\mathsf{a}(\xi)\quad \text{and}\quad \mathcal{T}_\Gamma\Ac_{\mathsf{f}}\mathsf{a}(\xi)= \mathbf{F}^*(\xi) \, \mathcal{T}_\Gamma \mathsf{a}(\xi),\quad \text{a.e.}\,\,\, \xi\in \widehat{N}\,.
\]
Besides, on the assumption that  $\mathsf{b}=\mathsf{a}\ast \mathsf{f}$ then
\begin{equation}
\label{pro}
\mathbf{B}(\xi)=\mathbf{F}(\xi)\mathbf{A}(\xi),\quad \text{a.e.}\,\,\, \xi\in \widehat{N}\,,
\end{equation}
where $\mathbf{A}(\xi)$ and $\mathbf{B}(\xi)$ are the transfer matrices associated to $\mathsf{a}$ and $\mathsf{b}$ defined in \eqref{F}.
 \end{teo}
\begin{proof}
By taking the $N$-Fourier transform in equalities  \eqref{sintesis} and  \eqref{analisis}, and having in mind Lemma \ref{conv} we obtain  that
$
\mathcal{T}_\Gamma\Lambda_{\mathsf{f}}\mathsf{a}(\xi)\, =\, \mathbf{F}(\xi) \, \mathcal{T}_\Gamma\mathsf{a}(\xi)\quad \text{and}\quad \mathcal{T}_\Gamma\Ac_{\mathsf{f}}\mathsf{a}(\xi)\, =\, \mathbf{F}^*(\xi) \, \mathcal{T}_\Gamma \mathsf{a}(\xi)\,. 
$
Concerning the second part, the $l$-column of $\mathbf{A}(\xi)$ is  $\mathcal{T}_\Gamma(L_{(0_N,l)}\mathsf{a})$ and
the $l$-column of $\mathbf{B}(\xi)$ is
\[
\mathcal{T}_\Gamma(L_{(0_N,l)}\mathsf{b})(\xi)= \mathcal{T}_\Gamma(L_{(0_N,l)}[\mathsf{a}\ast\mathsf{f}])(\xi)=\mathcal{T}_\Gamma([L_{(0_N,l)}\mathsf{a}]\ast\mathsf{f})(\xi)=\mathbf{F}(\xi)\mathcal{T}_\Gamma(L_{(0_N,l)}\mathsf{a})(\xi)\,,
\]
that is, the $l$-column of $\mathbf{F}(\xi)\mathbf{A}(\xi)$.
\end{proof}

\medskip

A similar formula to \eqref{pro} is obtained in Ref.~\cite{taylor:89} for functions in the $C^*$-algebra of the group $\Gamma$, even for a non discrete group $N$; see Section \ref{section5} below. It can be also found in Ref.~\cite{manning:12} for functions in $\ell^1(\Gamma)$, being $\Gamma$ a crystallographic group.
\section{Riesz bases for $\ell^2(\Gamma)$  generated by the left regular representation of $\Gamma$}
\label{section3}
In what follows we will assume that the non abelian group $\Gamma$ is the semidirect product $N\rtimes_\sigma H$, where $(N,+)$ is a discrete abelian group and 
$(H,\cdot)$ is a finite group of order $\kappa$. For a fixed $\mathsf{f}\in \ell^2(\Gamma)$, this section is devoted to give a characterization of the sequence $\big\{L_\gamma \mathsf{f}\big\}_{\gamma \in \Gamma}$ as a Riesz basis for $\ell^2(\Gamma)$ in terms of the associated matrix-valued function $\mathbf{F}(\xi)$ introduced in \eqref{F}.

\begin{teo}
\label{bessel}
For $\mathsf{f}\in \ell^2(\Gamma)$, let $\mathbf{F}(\xi)$ be its associated transfer matrix  defined in \eqref{F}. Then, the system
$\big\{L_{\gamma}\mathsf{f}\big\}_{\gamma\in \Gamma}$ is a Bessel sequence for $\ell^2(\Gamma)$ if and only if the entries of the matrix-valued function 
$\mathbf{F}(\xi)$ belong to $L^\infty(\widehat{N})$.  In this case the optimal Bessel bound is given by
\[ 
B_{\mathsf{f}}:=\esup_{\xi\in \widehat{N}} \lambda_{\text{max}}[\mathbf{F}^*(\xi) \mathbf{F}(\xi)]\,,
\]
where $\lambda_{\text{max}}$ denotes the largest eigenvalue of $\mathbf{F}^*(\xi) \mathbf{F}(\xi)$.
\end{teo}
 \begin{proof}
 Having in mind the equivalence between the spectral and the Frobenius norms for matrices \cite{horn:99}, we deduce that $B_{\mathsf{f}}<\infty$ if and only if the entries $\widehat{\mathsf{f}}_{h,l}$ of $\mathbf{F}(\xi)$ belong to $L^\infty(\widehat{N})$.

Whenever $\mathsf{a} \in \ell^2(N)$ with $\widehat{\mathsf{a}}_{l}(\xi)\,\widehat{\mathsf{f}}_{h,h'}(\xi)\in L^2(\widehat{N})$,\, $l,h,h'\in H$, by using Lemma \ref{T} and Theorem \ref{Tdomain}  we obtain that
\[
\Big\| \sum_{\gamma\in \Gamma} \mathsf{a}(\gamma)L_{\gamma}\mathsf{f}\Big\|_{\ell^2(\Gamma)}^2 = \big\| \Lambda_{\mathsf{f}}\mathsf{a} \big\|_{\ell^2(\Gamma)}^2=  \big\| \mathcal{T}_\Gamma\Lambda_{\mathsf{f}}\mathsf{a} \big\|_{L^2_\kappa(\widehat{N})}^2=
\Big\| \mathbf{F}(\cdot) \, \mathcal{T}_\Gamma\mathsf{a}(\cdot) \Big\|_{L^2_\kappa(\widehat{N})}^2\,,
\] 
and then,
\begin{equation}
\label{ce}
\Big\|\sum_{\gamma\in \Gamma} \mathsf{a}(\gamma) L_{\gamma}\mathsf{f}\Big\|^2_{\ell^2(\Gamma)}=  \int_{\widehat{N}} [\mathcal{T}_\Gamma\mathsf{a}(\xi)]^* \mathbf{F}^*(\xi) \mathbf{F}(\xi)\mathcal{T}_\Gamma\mathsf{a}(\xi)
 d\mu_{\widehat{N}}(\xi)\,.
\end{equation}
Hence, using Lemma \ref{T}, we obtain that, whenever $\mathsf{a} \in \ell^2(N)$ and $\widehat{\mathsf{f}}_{h,l}(\xi)\in L^\infty(\widehat{N})$, $l,h\in H$, we have
\[
\begin{split}
\Big\|\sum_{\gamma\in \Gamma} \mathsf{a}(\gamma) L_{\gamma}\mathsf{f}(\gamma)\Big\|^2_{\ell^2(\Gamma)}
& \le \int_{\widehat{N}} \lambda_{\text{max}}[\mathbf{F}^*(\xi) \mathbf{F}(\xi)]\,  \|\mathcal{T}_\Gamma\mathsf{a}(\xi)\|^2d\xi \\
 &\le B_{\mathsf{f}} \int_{\widehat{N}}   \|\mathcal{T}_\Gamma\mathsf{a}(\xi)\|^2d\xi= B_{\mathsf{f}} \, \|\mathcal{T}_\Gamma\mathsf{a}\|^2_{ L^2_\kappa(\widehat{N})}=
B_{\mathsf{f}} \, \|\mathsf{a}\|^2_{\ell^2(\Gamma)}\,. 
\end{split}
\]
Consequently, if  $\widehat{\mathsf{f}}_{h,l}(\xi)\in L^\infty(\widehat{N})$, $h,l\in H$, or equivalently $B_{\mathsf{f}}<\infty$, then $\big\{L_{\gamma}\mathsf{f}\big\}_{\gamma\in \Gamma}$ is a Bessel sequence with bound $B_{\mathsf{f}}$. 

\medskip

For any number $J<B_{\mathsf{f}}$, there exists a subset  $\Omega \subset \widehat{N}$ with positive measure such that $ \lambda_{\text{max}}[\mathbf{F}^*(\xi) \mathbf{F}(\xi)]>J$ for $\xi\in \Omega$. Let $\mathsf{a}\in \ell^2(\Gamma)$ such that $\mathcal{T}_\Gamma\mathsf{a}(\xi)$ is
 equal to $0$ when $\xi\notin \Omega$, and it is equal to a unitary eigenvector of 
 $\mathbf{F}^*(\xi)\mathbf{F}(\xi)$ corresponding to  the eigenvalue $ \lambda_{\text{max}}[\mathbf{F}^*(\xi) \mathbf{F}(\xi)]$  when $\xi\in\Omega$. By using \eqref{ce} we obtain
\[
\Big\|\sum_{\gamma\in \Gamma} \mathsf{a}(\gamma) L_{\gamma}\mathsf{f}(\gamma)\Big\|^2_{\ell^2(\Gamma)} =
\int_{\widehat{N}} \lambda_{\text{max}}[\mathbf{F}^*(\xi) \mathbf{F}(\xi)] \|\mathcal{T}_\Gamma \mathsf{a}(\xi)\|^2d\xi \ge J
\int_{\widehat{N}}   \|\mathcal{T}_\Gamma \mathsf{a}(\xi)\|^2d\xi=J \|\mathsf{a}\|^2_{\ell^2(\Gamma)}\,.
\]
Therefore, if $\big\{L_{\gamma}\mathsf{f}\big\}_{\gamma\in \Gamma}$ is Bessel sequence then $B_{\mathsf{f}}<\infty$, or equivalently $\widehat{\mathsf{f}}_{h,l}(\xi)\in L^\infty(\widehat{N})$, $h,l\in H$.  Moreover, the constant $B_{\mathsf{f}}$ is the optimal Bessel bound. 
\end{proof}

\medskip

\begin{teo}
\label{riesz}
Consider $\mathsf{f}\in \ell^2(\Gamma)$ and its associated transfer matrix $\mathbf{F}(\xi)$ given in \eqref{F}. The following statements are equivalent:
\begin{enumerate}[(a)]
\item The system $\big\{L_{\gamma}\mathsf{f}\big\}_{\gamma\in \Gamma}$ is a Riesz basis for $\ell^2(\Gamma)$. 
\item The entries of  the matrix-valued function $\mathbf{F}(\xi)$ belongs to $L^\infty(\widehat{N})$, and $\displaystyle{\einf_{\xi\in \widehat{N}} |\det \mathbf{F}(\xi)|>0}$.
\end{enumerate}
In this case the optimal Riesz bounds are given by
\[
A_{\mathsf{f}}:=\einf_{\xi\in \widehat{N}} \lambda_{\text{min}}[\mathbf{F}^*(\xi) \mathbf{F}(\xi)]\quad \text{and}\quad 
B_{\mathsf{f}}:=\esup_{\xi\in \widehat{N}} \lambda_{\text{max}}[\mathbf{F}^*(\xi) \mathbf{F}(\xi)]\,.
\]
Moreover, its dual Riesz basis is $\big\{L_{\gamma}\mathsf{g}\big\}_{\gamma\in \Gamma}$ where $\mathsf{g}$ is the unique element in $\ell^2(\Gamma)$ satisfying
\begin{equation}
\label{dual1}
\mathbf{F}^*(\xi) \mathcal{T}_\Gamma\mathsf{g}(\xi)=\begin{bmatrix}1&0&\ldots&0\end{bmatrix}^\top,\quad \text{a.e.} \,\, \xi \in \widehat{N},
\end{equation}
where $\mathcal{T}_\Gamma$ is defined in Lemma \ref{T}.
Equivalently, $\mathbf{G}(\xi)=(\mathbf{F}^*(\xi))^{-1}$,\, a.e. \,\,$\xi \in \widehat{N}$ where $\mathbf{G}(\xi)$ is the transfer matrix associated to $\mathsf{g}$.
\end{teo}
\begin{proof}
First of all, note that for any $\kappa\times \kappa$ hermitian matrix $\mathbf{M}$ we have
 that \[\lambda^\kappa_{\text{min}}(\mathbf{M})\le \det \mathbf{M}=\lambda_{\text{min}}(\mathbf{M}) \cdots \cdot\lambda_{\text{max}}(\mathbf{M})
 \le \lambda_{\text{min}}(\mathbf{M}) \lambda^{\kappa-1}_{\text{max}}(\mathbf{M})\,.\]
Using these inequalities for $\mathbf{M}=\mathbf{F}^*(\xi)\mathbf{F}(\xi)$ we obtain that
 \begin{equation}
 \label{det}
 A_{\mathsf{f}}^\kappa\le \einf_{\xi\in \widehat{N}} |\det \mathbf{F}(\xi)|^2\le A_{\mathsf{f}} B_{\mathsf{f}}^{\kappa-1}\,.
 \end{equation}	
 
\noindent$(a)\Rightarrow (b)$. If $(a)$ holds $\big\{L_{\gamma}\mathsf{f}\big\}_{\gamma\in \Gamma}$ is a Bessel system; thus, having in mind Theorem \ref{bessel}, the entries of  the matrix-valued function $\mathbf{F}(\xi)$ belong to $L^\infty(\widehat{N})$.
Using the Rayleigh-Ritz theorem \cite{horn:99} and  Lemma \ref{T},  for any $\mathsf{a} \in \ell^2(N)$ we obtain that 
\begin{equation}
\label{lowerriesz}
\begin{split}
\Big\|\sum_{\gamma\in \Gamma} \mathsf{a}(\gamma) L_{\gamma}\mathsf{f}\Big\|^2_{\ell^2(\Gamma)}&=  \int_{\widehat{N}} [\mathcal{T}_\Gamma\mathsf{a}(\xi)]^* \mathbf{F}^*(\xi) \mathbf{F}(\xi)\mathcal{T}_\Gamma\mathsf{a}(\xi)
 d\mu_{\widehat{N}}(\xi)\\
 &\ge \int_{\widehat{N}}  
 \lambda_{\text{min}}[\mathbf{F}^*(\xi) \mathbf{F}(\xi)]\, \|\mathcal{T}_\Gamma\mathsf{a}(\xi)\|^2d\xi \\&\ge A_{\mathsf{f}}
\int_{\widehat{N}}   \|\mathcal{T}_\Gamma\mathsf{a}(\xi)\|^2d\xi= A_{\mathsf{f}} \, \|\mathcal{T}_\Gamma\mathsf{a}\|^2_{ L^2_\kappa(\widehat{N})}=
A_{\mathsf{f}} \, \|\mathsf{a}\|^2_{\ell^2(\Gamma)}.
  \end{split}
\end{equation}
  For any number $J>A_{\mathsf{f}}$, there exist a subset  $\Omega \subset \widehat{N}$ with positive measure such that $ \lambda_{\text{min}}[\mathbf{F}^*(\xi) \mathbf{F}(\xi)]<J$ for $\xi\in \Omega$. Let $\mathsf{a}\in \ell^2(\Gamma)$ such that $\mathcal{T}_\Gamma\mathsf{a}(\xi)$
 is equal to $0$ when $\xi\notin \Omega$, and it is equal to a unitary eigenvector of 
 $\mathbf{F}^*(\xi)\mathbf{F}(\xi)$ corresponding to  the eigenvalue $ \lambda_{\text{min}}[\mathbf{F}^*(\xi) \mathbf{F}(\xi)]$  when $\xi\in\Omega$. Then, using \eqref{ce}, we obtain
\[
\begin{split}
&
\Big\|\sum_{\gamma\in \Gamma} \mathsf{a}(\gamma) L_{\gamma}\mathsf{f}(\gamma)\Big\|^2_{\ell^2(\Gamma)} =
\int_{\widehat{N}}  
 \lambda_{\text{min}}[\mathbf{F}^*(\xi) \mathbf{F}(\xi)] \|\mathcal{T}_\Gamma\mathsf{a}(\xi)\|^2d\xi \le J
\int_{\widehat{N}}   \|\mathcal{T}_\Gamma\mathsf{a}(\xi)\|^2d\xi=
J \|\mathsf{a}\|^2_{\ell^2(\Gamma)}.
  \end{split}
\]
Therefore, if $\big\{L_{\gamma}\mathsf{f}\big\}_{\gamma\in \Gamma}$ is a Riesz basis then $A_{\mathsf{f}}>0$ which,  having in mind \eqref{det}, implies $\displaystyle{\einf_{\xi\in \widehat{N}} |\det \mathbf{F}(\xi)|>0}$. Besides, the optimal lower Riesz bound is  $A_{\mathsf{f}}$. 

\medskip

\noindent$(b)\Rightarrow (a)$. Since  the entries of  $\mathbf{F}(\xi)\in L^\infty(\widehat{N})$, using Theorem \ref{bessel} we deduce that $B_{\mathsf{f}}<\infty$. Since  $\displaystyle{\einf_{\xi\in \widehat{N}} |\det \mathbf{F}(\xi)|>0}$, using \eqref{det} we obtain $A_{\mathsf{f}}>0$. As a consequence of Theorem \ref{bessel} and inequality \eqref{lowerriesz} we deduce that  the system $\big\{L_{\gamma}\mathsf{f}\big\}_{\gamma\in \Gamma}$ is a Riesz basis for $\ell^2(\Gamma)$.

\medskip

Next, we find its dual Riesz basis.  Since $\einf_{\xi\in \widehat{N}} |\det \mathbf{F}(\xi)|>0$, and the entries of $\mathbf{F}^*(\xi)$ belongs to $L^\infty(\widehat{N})$, the entries of the matrix-valued function
$[\mathbf{F}^*(\xi)]^{-1}$ belong to $\ell^2(\widehat{N})$. By Lemma \ref{T}, there exist a unique $\mathsf{g}\in \ell^2(\Gamma)$ such that $\mathcal{T}_\Gamma\mathsf{g}$ is the first column of $[\mathbf{F}^*(\xi)]^{-1}$, or equivalently, a unique $\mathsf{g}\in \ell^2(\Gamma)$ satisfying \eqref{dual1}. From Theorem \ref{Tdomain} we get

\[
\mathcal{T}_\Gamma \big[\langle \mathsf{g},L_{\gamma}\mathsf{f} \rangle\big]_{\gamma\in \Gamma}(\xi)= \mathcal{T}_\Gamma(\Ac_{\mathsf{f}}\mathsf{g})(\xi)= \mathbf{F}^*(\xi) \mathcal{T}_\Gamma\mathsf{g}(\xi)=
\begin{bmatrix}1&0&\ldots&0\end{bmatrix}^\top
\]
Hence $\langle \mathsf{g},L_{\gamma}\mathsf{f} \rangle=\boldsymbol{\delta}(\gamma)$, and the system $\big\{L_{\gamma}\mathsf{g}\big\}_{\gamma\in \Gamma}$ is the dual Riesz basis to $\big\{L_{\gamma}\mathsf{f}\big\}_{\gamma\in \Gamma}$\,. Besides, we have that $\mathsf{g}\ast \mathsf{f}^*(\gamma)=\langle \mathsf{g},L_{\gamma}\mathsf{f} \rangle=\boldsymbol{\delta}(\gamma)$.  Applying \eqref{pro},  having in mind that the matrices corresponding to $\boldsymbol{\delta}$ and $\mathsf{f}^*$ are $\big[\, \widehat{\boldsymbol{\delta}}_{h,l}\,\big]_{h,l\in H}=\mathbf{I}_\kappa$ and $\big[\, \widehat{\mathsf{f}^*_{h,l}}\,\big]_{h,l\in H}=\mathbf{F}^*$ we get that 
 $\mathbf{F}^*(\xi)\mathbf{G}(\xi)=\mathbf{I}_\kappa $,\, a.e. $\xi \in \widehat{N}$.
\end{proof}
 
 \begin{cor}
 \label{ort}  
The system $\big\{L_{\gamma}\mathsf{f}\big\}_{\gamma\in \Gamma}$
is an orthonormal basis for $\ell^2(\Gamma)$ if and only if the matrix-valued function $\mathbf{F}(\xi)$ is unitary a.e. $\xi \in \widehat{N}$, or equivalently, if $\mathbf{F}^*(\xi) \mathcal{T}_\Gamma\mathsf{f}(\xi)=\begin{bmatrix}1&0&\ldots&0\end{bmatrix}^\top$ a.e. $\xi \in \widehat{N}$.
\end{cor}
\begin{proof}
The Riesz basis $\big\{L_{\gamma}\mathsf{f}\big\}_{\gamma\in \Gamma}$ is an ortonormal basis if and only if the generator of the dual Riesz basis $\mathsf{g}=\mathsf{f}$, or equivalently, if $\mathbf{G}=\mathbf{F}$. Thus, the result follows from Theorem \ref{riesz}.
\end{proof} 

\begin{cor}
\label{dual} 
Consider $\mathsf{f},\mathsf{g}\in \ell^2(\Gamma)$ and their associated transfer matrices $\mathbf{F}(\xi)$, $\mathbf{G}(\xi)$ defined in \eqref{F}. Assume that  the entries of $\mathbf{F}(\xi)$ and $\mathbf{G}(\xi)$ belong to $L^\infty(\widehat{N})$. Then,  the systems $\big\{L_{\gamma}\mathsf{f}\big\}_{\gamma\in \Gamma}$ and $\big\{L_{\gamma}\mathsf{g}\big\}_{\gamma\in \Gamma}$  form a pair of dual Riesz bases if and only if  $\mathbf{G}(\xi)=\big[\mathbf{F}^*(\xi)\big]^{-1}$ a.e. $\xi\in \widehat{N}$.
\end{cor}
\begin{proof} 
Since the entries of $\mathbf{G}(\xi)$ belong to $L^\infty(\widehat{N})$, if $\mathbf{G}(\xi)\mathbf{F}^*(\xi)=\mathbf{I}_{\kappa}$ we have that  
\[
\einf_{\xi\in \widehat{N}} |\det \mathbf{F}(\xi)|= \einf_{\xi\in \widehat{N}} \big(|\det \mathbf{G}(\xi)|^{-1} \big)=
 \big(\esup_{\xi\in \widehat{N}} |\det \mathbf{G}(\xi)| \big)^{-1}>0.
 \] 
 Hence, the result is easily obtained from Theorem \ref{riesz}.
 \end{proof} 
\subsection{Remarks}
\label{remark}
\noindent $\bullet$ Whenever the generator $\mathsf{f}$ belongs to $\ell^1(\Gamma)$, the matrix-valued function $\mathbf{F}(\xi)$ has continuous entries in the compact $\widehat{N}$ (recall that $N$ is discrete).  Thus, from Theorem \ref{bessel} the system $\big\{L_{\gamma}\mathsf{f}\big\}_{\gamma\in \Gamma}$ is always a Bessel sequence for 
$\ell^2(\Gamma)$. From Theorem \ref{riesz}  it is a Riesz basis for $\ell^2(\Gamma)$ if and only if the matrix-valued function $\mathbf{F}(\xi)$ is non-singular for all 
$\xi\in \widehat{N}$.
Finally, from Corollary \ref{ort}  it is an orthonormal basis if and only if $\mathbf{F}(\xi)$ is unitary for all $\xi\in \widehat{N}$.

\medskip

\noindent $\bullet$  Theorems \ref{bessel} and \ref{riesz}, and Corollary \ref{ort} can be restated in terms of the convolution operator.
 Namely (see \cite[Lemma 3.2.1 and Proposition 3.6.8]{ole:16}),
 \begin{enumerate}[a.]
 \item The expression $\Lambda_{\mathsf{f}}\mathsf{a}=\sum_{\eta\in \Gamma} \mathsf{a}(\eta)L_{\eta}\mathsf{f}=\mathsf{a}\ast \mathsf{f}$
defines a  bounded linear operator $\Lambda_{\mathsf{f}}:\ell^2(\Gamma) \rightarrow  \ell^2(\Gamma)$  if and only if the entries of the matrix-valued function $\mathbf{F}(\xi)$ belong to $L^\infty(\widehat{N})$. In this case, $\|\Lambda_{\mathsf{f}}\|=B^{1/2}_{\mathsf{f}}$.

\item Assume that the entries of the matrix-valued function $\mathbf{F}(\xi)$ belong to $L^\infty(\widehat{N})$. Then, $\Lambda_{\mathsf{f}}$ is an invertible operator  if and only if  $\einf_{\xi\in \widehat{N}} |\det \mathbf{F}(\xi)|>0$.  In this case, $\|\Lambda_{\mathsf{f}}^{-1}\|=A^{-1/2}_{\mathsf{f}}$ and the condition number of 
$\Lambda_{\mathsf{f}}$ is $\|\Lambda_{\mathsf{f}}\|\, \|\Lambda_{\mathsf{f}}^{-1}\|=(B_{\mathsf{f}}/A_{\mathsf{f}})^{1/2}$. Besides, there exists a unique 
$\mathsf{h}\in \ell^2(\Gamma)$ such that $\mathsf{f}\ast \mathsf{h}=\boldsymbol{\delta}$, and satisfying
\[
 \mathbf{H}(\xi) \mathbf{F}(\xi) = \mathbf{I}_\kappa\,,\quad \text{a.e.} \,\, \xi \in \widehat{N}\,,
\]
where $\mathbf{H}(\xi)$ is the transfer matrix associated to $\mathsf{h}$. Besides, $\mathsf{h}=\mathsf{g}^*$ where $\mathsf{g}$ is the function defined in Theorem \ref{riesz}.

\item The linear map $\Lambda_{\mathsf{f}}$ defines a unitary operator  if and only if $\mathbf{F}(\xi)$ is unitary  a.e. $\xi \in \widehat{N}$.
\end{enumerate}
\subsection{Riesz bases examples associated to the infinite dihedral group $D_{\infty}$}
We illustrate the above results of this section in  the case of the infinite dihedral group $D_{\infty}= \ZZ \rtimes \{1,-1\}$. Recall  that $\widehat{\mathbb{Z}}\cong \mathbb{T}$, with $\langle n,z \rangle = z^n$, $z\in \mathbb{T}$, and consequently the Fourier transform of  the sequence $\{\mathsf{a}(n)\}_{n\in \mathbb{Z}}$ is the $z$-transform  $\widehat{\mathsf{a}}(z)=\sum_{n\in \mathbb{Z}} \mathsf{a}(n) z^{-n}$ (see, for instance, \cite[Theorem 4.5]{folland:95}). 

For any $\mathsf{f} \in \ell^2(D_{\infty})$,
the first column of $\mathbf{F}(z),$ is formed by the $z$-transforms of
$\mathsf{f}_{1}(n)=\mathsf{f}(n,1)$ and $\mathsf{f}_{-1}(n)=\mathsf{f}(n,-1)$, $n\in \ZZ$.
The second column is formed by the $z$-transforms of $\mathsf{f}_{1,-1}(n)=\mathsf{f}_{-1}(-n)$ and $\mathsf{f}_{-1,-1}(n)=\mathsf{f}_{1}(-n)$, $n\in \ZZ$; that is
\begin{equation}\label{ej}
\mathbf{F}(z)=\begin{bmatrix}\widehat{\mathsf{f}}_{1}(z)&\widehat{\mathsf{f}}_{-1}(z^{-1})\\
\widehat{\mathsf{f}}_{-1}(z)&\widehat{\mathsf{f}}_{1}(z^{-1})\end{bmatrix}\,,\quad z\in \mathbb{T}\,.
\end{equation}

Firstly, according to Corollary \ref{ort}, the sequence $\{L_{\gamma}\mathsf{f}\}_{\gamma\in D_{\infty}}$ is an orthonormal basis for $\ell^2(D_{\infty})$ if and only if
$\mathbf{F}^*(z)\mathcal{T}_{\Gamma}\mathsf{f}(z)=\begin{bmatrix}1&0\end{bmatrix}^\top$, or equivalently
\[
\begin{split}
&|\widehat{\mathsf{f}}_{1}(z)|^2+|\widehat{\mathsf{f}}_{-1}(z)|^2=1\quad \text{and}\quad\widehat{\mathsf{f}}_{1}(z)\overline{\widehat{\mathsf{f}}_{-1}(z^{-1})}+\widehat{\mathsf{f}}_{-1}(z)\overline{\widehat{\mathsf{f}}_{1}(z^{-1})}=0\,,\quad \text{a.e.} \, \, z\in \mathbb{T}\,.
\end{split}
\]
These equations are satisfied, for example, when  $\widehat{\mathsf{f}}_{-1}(z)=0$ and $|\widehat{\mathsf{f}}_{1}(z)|=1$. In signal processing jargon, complex transfer functions satisfying $|f(z)|=1$ in $\mathbb{T}$ are called allpass filters; expressions for rational allpass filter, their properties, as well as efficient ways to compute the 
$\ast_{N}$-convolutions in \eqref{sintesis} and \eqref{analisis}, can be found in Ref.~ \cite[Section 3.4]{vaid:93}. The 
simplest allpass filter $\widehat{\mathsf{f}}_{1}(z)=z^k$ yields to the canonical  basis $\{L_{\gamma}\boldsymbol{\delta}\}_{\gamma\in D_{\infty}}$. Other interesting solutions of the above equations are 
\[
\widehat{\mathsf{f}}_{1}(e^{iw})=\begin{cases}1,&|w|\le a\\
0,&|w|> a
\end{cases}\,;\quad \widehat{\mathsf{f}}_{-1}(e^{iw})=\begin{cases}0,&|w|\le a\\
1,&|w|> a
\end{cases}
\]
for a fixed $a\in (0,\pi)$.
\medskip

Secondly, according to Theorem \ref{riesz}, the sequence $\{L_{\gamma}\mathsf{f}\}_{\gamma\in D_{\infty}}$ is a Riesz basis for $\ell^2(D_{\infty})$ if and only if
\[
\einf_{z\in \mathbb{T}} |\det \mathbf{F}(z)|= \einf_{z\in \mathbb{T}} \big|\widehat{\mathsf{f}}_{1}(z)\widehat{\mathsf{f}}_{1}(z^{-1})-\widehat{\mathsf{f}}_{-1}(z)\widehat{\mathsf{f}}_{-1}(z^{-1})\big|>0\,.\]
In this case, by solving 
$\mathbf{F}^*(z) \begin{bmatrix}\widehat{\mathsf{g}}_{1}(z)&
\widehat{\mathsf{g}}_{-1}(z)\end{bmatrix}^\top=
\begin{bmatrix}1&
0\end{bmatrix}^\top
$, we obtain
\begin{equation}\label{ej1}
\begin{split}
&\widehat{\mathsf{g}}_{1}(z)=\frac{\overline{\widehat{\mathsf{f}}_{1}(z)}}{\det \mathbf{F}^*(z)}\,;\quad \widehat{\mathsf{g}}_{-1}(z)=\frac{-\overline{\widehat{\mathsf{f}}_{-1}(z)}}{\det \mathbf{F^*}(z)}\,,
\end{split}
\end{equation}
which provides the generator $\mathsf{g}$ of its dual Riesz basis $\{L_{\gamma}\mathsf{g}\}_{\gamma\in D_{\infty}}$. Whether $\mathsf{f}$ is real,
we have that  $\widehat{\mathsf{f}}_{1}(z^{-1})=\overline{\widehat{\mathsf{f}}_{1}(z)}$ and $\widehat{\mathsf{f}}_{-1}(z^{-1})=\overline{\widehat{\mathsf{f}}_{-1}(z)}$, from which it is straightforward to deduce that the optimal Riesz bounds are
\begin{equation}\label{rb}
\begin{split}
&A_{\mathsf{f}}:=\einf_{z\in \mathbb{T}} \lambda_{\text{min}}[\mathbf{F}^*(z) \mathbf{F}(z)]=\einf_{w\in [-\pi,\pi]} 
\big(|\widehat{\mathsf{f}}_{1}(e^{iw})|- |\widehat{\mathsf{f}}_{-1}(e^{iw})|\big)^2,
\\
&B_{\mathsf{f}}:=\esup_{z\in \mathbb{T}} \lambda_{\text{max}}[\mathbf{F}^*(z) \mathbf{F}(z)]=\esup_{w\in [-\pi,\pi]}\big(|\widehat{\mathsf{f}}_{1}(e^{iw})|+|\widehat{\mathsf{f}}_{-1}(e^{iw})|\big)^2\,.
\end{split}
\end{equation}

Closing this section we exhibit a simple example. The generators $\mathsf{f}$ and $\mathsf{g}$  are  finitely supported  whenever  $\widehat{\mathsf{f}}_{1}$ and $\widehat{\mathsf{f}}_{-1}$ are Laurent polynomials such that 
$\det \mathbf{F}(z)=\widehat{\mathsf{f}}_{1}(z)\widehat{\mathsf{f}}_{1}(z^{-1})-\widehat{\mathsf{f}}_{-1}(z)\widehat{\mathsf{f}}_{-1}(z^{-1})=a z^k$ for some $k\in \ZZ$, $a\neq 0$. 
For instance, for $\widehat{\mathsf{f}}_{1}(z)=3/8$ and $\widehat{\mathsf{f}}_{-1}(z)=z/8$ we obtain $\widehat{\mathsf{g}}_{1}(z)=3$ and $\widehat{\mathsf{g}}_{-1}(z)=-z$; thus the dual generators $\mathsf{f}$ and $\mathsf{g}$ have both support of size 2. 
From \eqref{rb}, the optimal Riesz bounds of $\big\{L_{\gamma}\mathsf{f}\big\}_{\gamma\in \Gamma}$ are $A_{\mathsf{f}}=1/4$ and $B_{\mathsf{f}}=1/2$.

\section{A sampling application}
\label{section4}
Let $\Gamma \ni \gamma \longmapsto  U(\gamma) \in \mathcal{U}(\mathcal{H})$ be a unitary representation of the group $\Gamma$ in a separable Hilbert space 
$\mathcal{H}$, i.e., a homomorphism between $\Gamma$ and $\mathcal{U}(\mathcal{H})$. We are interested in the study of sampling in the principal $U$-invariant space $\Ac_\varphi:=\overline{\espan}\{U(\gamma)\varphi\}_{\gamma \in \Gamma}$ of $\Hc$, where $\varphi$ denotes a fixed element of $\Hc$. In case the sequence $\{U(\gamma)\varphi\}_{\gamma\in \Gamma}$ is a Riesz sequence for $\Hc$ (one can find necessary and sufficient conditions in Refs.~\cite{barbieri:15,manning:12}) the subspace $\Ac_\varphi$ can be expressed as
 \[
\Ac_\varphi= \Big\{ \sum_{\gamma\in \Gamma} \mathsf{a}(\gamma) U(\gamma)\varphi \,:\,
\mathsf{a}=\{\mathsf{a}(\gamma)\}_{\gammaÊ\in \Gamma} \in \ell^2(\Gamma)\Big\}
\]
The goal here is the stable recovery of any $f\in \Ac_\varphi$ from the data $\{\Lc_\psi f(\gamma)\}_{\gamma \in \Gamma}$ given by
\begin{equation}
\label{samples}
\mathcal{L}_\psi f(\gamma):=\big\langle f, U(\gamma)\psi \big\rangle_\mathcal{H}\,,\quad \gamma\in \Gamma\,,
\end{equation}
where $\psi\in \Hc$ is a fixed element which does not belong necessarily to $\Ac_\varphi$. First, we express the samples in a more suitable manner; for each  $f=\sum_{\gamma\in \Gamma} \mathsf{a}(\gamma) U(\gamma)\varphi$ in $\Ac_\varphi$ we have
\[
\begin{split}
\mathcal{L}_\psi f(\gamma)&=\Big\langle \sum_{\eta\in \Gamma} \mathsf{a}(\eta) U(\eta)\varphi, U(\gamma)\psi \Big\rangle_\mathcal{H}= \sum_{\eta\in \Gamma}  \mathsf{a}(\eta) \big\langle  \varphi, U(\eta^{-1}\gamma)\psi \big\rangle_\mathcal{H} \\
&=\sum_{\eta\in \Gamma}  \mathsf{a}(\eta) \overline{ \mathsf{f}(\gamma^{-1}\eta)}=
\sum_{\eta\in \Gamma}  \mathsf{a}(\eta) \overline{L_\gamma\mathsf{f}(\eta)} =\big\langle \mathsf{a}, L_\gamma\mathsf{f}_\psi \big\rangle_{\ell^2(\Gamma)}\,,\quad \gamma \in \Gamma\,,
\end{split}
\]
where $\mathsf{f}_\psi(\eta):=\langle \overline{\varphi, U(\eta^{-1}) \psi} \rangle_\Hc$ for $\eta \in \Gamma$, and $\mathsf{a}=\{\mathsf{a}(\gamma)Ê\}_{\gammaÊ\in \Gamma}$. Notice that $\mathsf{f}_\psi$ belongs to $\ell^2(\Gamma)$. In the light of Theorem \ref{riesz}, assume that 
$\{L_\gamma\mathsf{f}_\psi\}_{\ell^2(\Gamma)}$ is a Riesz basis for $\ell^2(\Gamma)$ with dual Riesz basis $\{L_\gamma\mathsf{g}_\psi\}_{\ell^2(\Gamma)}$. Thus, for any $\mathsf{a}\in \ell^2(\Gamma)$ we have 
\begin{equation}
\label{expa}
\mathsf{a}=\sum_{\gamma\in \Gamma} \langle \mathsf{a}, L_\gamma\mathsf{f}_\psi \rangle_{\ell^2(\Gamma)}\, L_\gamma\mathsf{g}_\psi=
\sum_{\gamma\in \Gamma} \mathcal{L}_\psi f(\gamma)\, L_\gamma\mathsf{g}_\psi\quad \text{ in $\ell^2(\Gamma)$}\,.
\end{equation}
In order to derive a sampling formula in $\Ac_\varphi$ compatible with its structure, we consider the natural isomorphism $\Tc_{U,\varphi} : \ell^2(\Gamma) \rightarrow  \Ac_\varphi$ which maps the usual orthonormal basis $\{\boldsymbol{\delta}_\gamma\}_{\gamma\in \Gamma}$ for $\ell^2(\Gamma)$ onto the Riesz basis $\{U(\gamma)\varphi\}_{\gamma\in \Gamma}$ for $\Ac_\varphi$. This isomorphism satisfies the following {\em shifting property}:
\[
\Tc_{U,\varphi} \big(L_\gamma\mathsf{f} \big)=U(\gamma) \Tc_{U,\varphi}\mathsf{f}\quad \text{for each $\mathsf{f}\in \ell^2(\Gamma)$ and $\gamma \in \Gamma$}\,.
\]
Indeed, we have that $L_\gamma \boldsymbol{\delta}_\eta=\boldsymbol{\delta}_{\gamma \eta}$ for $\gamma, \eta \in \Gamma$. As a consequence, 
$\Tc_{U,\varphi} \big(L_\gamma \boldsymbol{\delta}_\eta\big) =\Tc_{U,\varphi} \boldsymbol{\delta}_{\gamma \eta}= U(\gamma) U(\eta)\varphi= U(\gamma)\Tc_{U,\varphi}\big( \boldsymbol{\delta}_{\eta}\big)$.
From a  continuity argument the result becomes true for any $\mathsf{f}\in \ell^2(\Gamma)$. 

\medskip

Now, consider $f=\Tc_{U,\varphi}(\mathsf{a})$ in $\Ac_\varphi$;  applying the isomorphism $\Tc_{U,\varphi}$ in expansion \eqref{expa} and using the above shifting property we obtain for each $f\in \Ac_\varphi$ the sampling formula 
\begin{equation}
\label{sampexp}
f=\Tc_{U,\varphi}(\mathsf{a})=\sum_{\gamma\in \Gamma} \mathcal{L}_\psi f(\gamma)\, \Tc_{U,\varphi}\big(L_\gamma\mathsf{g}_\psi\big)=\sum_{\gamma \in \Gamma} \mathcal{L}_\psi f(\gamma)\, U(\gamma)\Tc_{U,\varphi}(\mathsf{g}_\psi)\quad \text{in $\Hc$}\,.
\end{equation}
Notice that $\Tc_{U,\varphi}(\mathsf{g}_\psi)=\sum_{\gamma \in \Gamma}\mathsf{g}_\psi(\gamma)\,U(\gamma)\varphi \in\Ac_\varphi$.
Moreover, since $\Tc_{U,\varphi}$ is an isomorphism, the sequence $\{U(\gamma)\Tc_{U,\varphi}(\mathsf{g}_\psi)\}_{\gamma \in \Gamma}$ is a Riesz basis for $\Ac_\varphi$. In fact, the following sampling theorem in $\Ac_\varphi$ holds:
\begin{teo}
\label{sampriesz}
For a given $\psi\in \Hc$, consider $\mathsf{f}_\psi \in \ell^2(\Gamma)$ such that $\mathsf{f}_\psi(\eta):=\langle \overline{\varphi, U(\eta^{-1}) \psi} \rangle_\Hc$ for $\eta \in \Gamma$. Assume that all the entries of its associated $\kappa \times \kappa$ matrix-valued function  $\mathbf{F}(\xi)$ defined in \eqref{F} belong to 
$L^\infty(\widehat{N})$. The following statements are equivalent:
\begin{enumerate}[(a)]
\item ${\displaystyle \einf_{\xi\in \widehat{N}} |\det \mathbf{F}(\xi)| >0}$.
\item There exists a unique $\mathsf{g}_\psi\in \ell^2(\Gamma)$  such that its associate matrix-valued function $\mathbf{G}(\xi)$ defined in \eqref{F} has entries in 
$L^\infty(\widehat{N})$, and it satisfies $\mathbf{G}(\xi)\mathbf{F}^*(\xi)=\mathbf{I}_\kappa$,\, a.e.  $\xi\in \widehat{N}$.
\item There exists a unique $\Phi_\psi\in \Ac_\varphi$ such that the sequence $\{U(\gamma)\Phi_\psi\}_{\gamma \in \Gamma}$ is a Riesz basis for $\Ac_\varphi$ and the sampling formula 
\begin{equation}
\label{sformula}
f=\sum_{\gamma \in \Gamma} \mathcal{L}_\psi f(\gamma)\, U(\gamma)\Phi_\psi \quad \text{ in $\Hc$}
\end{equation}
holds for each $f\in \Ac_\varphi$.
\end{enumerate}
In case the equivalent conditions are satisfied, necessarily $\Phi_\psi=\Tc_{U,\varphi}(\mathsf{g}_\psi)$ where $\mathsf{g}_\psi \in \ell^2(\Gamma)$ satisfies conditions in $(b)$. Moreover, the interpolation property $\Lc_\psi \Phi_\psi(\gamma)=\delta_{\gamma, e_\Gamma}$, $\gamma \in \Gamma$, holds.
\end{teo}
\begin{proof}

\noindent $(a)\Rightarrow(b)$. The sequence $\{L_\gamma\mathsf{f}_\psi\}_{\ell^2(\Gamma)}$ is a Riesz basis for $\ell^2(\Gamma)$. Having in mind Theorem \ref{riesz}, its dual Riesz basis has the form $\{L_\gamma\mathsf{g}_\psi\}_{\ell^2(\Gamma)}$ with $\mathbf{G}(\xi)\mathbf{F}^*(\xi)=\mathbf{I}_\kappa$,\, a.e.  
$\xi\in \widehat{N}$.

\noindent $(b)\Rightarrow(c)$. According with Corollary \ref{dual}, the sequences $\{L_\gamma\mathsf{f}_\psi\}_{\ell^2(\Gamma)}$ and $\{L_\gamma\mathsf{g}_\psi\}_{\ell^2(\Gamma)}$ form a pair of dual Riesz bases for $\ell^2(\Gamma)$. Thus we have \eqref{expa} and, consequently, \eqref{sampexp} proves $(c)$.

\noindent $(c)\Rightarrow(a)$. Applying the isomorphism $\Tc_{U,\varphi}^{-1}$, the sequence $\big\{\Tc_{U,\varphi}^{-1}\big( U(\gamma)\Phi\big)\big\}_{\gamma \in \Gamma}$ is a Riesz sequence for $\ell^2(\Gamma)$, and for each $\mathbf{a}\in \ell^2(\Gamma)$ we get 
\[
\mathsf{a}=\sum_{\gamma\in \Gamma} \langle \mathsf{a}, L_\gamma\mathsf{f}_\psi \rangle_{\ell^2(\Gamma)}\,\Tc_{U,\varphi}^{-1}\big( U(\gamma)\Phi_\psi\big)\quad \text{ in $\ell^2(\Gamma)$}\,.
\]
The sequence $\{L_\gamma\mathsf{f}_\psi\}_{\ell^2(\Gamma)}$ is a Bessel sequence biorthogonal to $\big\{\Tc_{U,\varphi}^{-1}\big( U(\gamma)\Phi_\psi\big)\big\}_{\gamma \in \Gamma}$, and hence it is a Riesz basis for $\ell^2(\Gamma)$ \cite[Theorem 3.6.7]{ole:16}; from Theorem \ref{riesz}, ${\displaystyle \einf_{\xi\in \widehat{N}} |\det\mathbf{F}(\xi)| >0}$.

The uniqueness of the coefficients in a Riesz basis expansion gives the interpolation property $\Lc_\psi \Phi_\psi(\gamma)=\delta_{\gamma, e_\Gamma}$, $\gamma \in \Gamma$.
\end{proof}
\subsection{The crystallographic group case}
The {\em euclidean motion group} $E(d)$ is the semidirect product $\RR^d \rtimes_{\sigma} O(d)$ corresponding to the homomorphism $\sigma : O(d) \rightarrow Aut(\RR^d)$ given by $\sigma(A)(x) = Ax$, where $A\in O(d)$ and $x\in \RR^d$. The composition law on $E(d) = \RR^d \rtimes_{\sigma} O(d)$ reads $(x, A) \cdot  (x', A') = (x + Ax', AA')$.

Let $\Mc$ be a non-singular $d\times d$ matrix and $H$ a finite subgroup of $O(d)$ of order $\kappa$ such that  $A(\Mc\ZZ^d)=\Mc\ZZ^d$ for each $A\in H$. We consider the {\em crystallographic group} $\Gamma_{\Mc,H}:=\Mc\ZZ^d \rtimes_\sigma H$ and its {\em quasi regular representation} (see, for instance, Ref.~\cite{barbieri:15}) on $L^2(\RR^d)$
\[
U(n,A)f(t)=f[A^{\top}(t-n)]\,,\quad \text{$n\in \Mc\ZZ^d$, $A\in H$ and $f\in L^2(\RR^d)$}\,.
\]
For a fixed $\varphi \in L^2(\RR^d)$ such that the sequence $\big\{ U(n,A)\varphi \big\}_{(n,A)\in \Gamma_{\Mc,H}}$ is a Riesz sequence for $L^2(\RR^d)$ (see, for instance, Refs.\cite{boor:93,jia:91}) we consider the $U$-invariant subspace in $L^2(\RR^d)$
\[
\Ac_\varphi=\Big\{\sum_{(n,A)\in \Gamma_{\Mc,H}} \alpha(n,A)\, \varphi [A^{\top}(t-n)] \,\,:\,\, \{\alpha(n,A)\}\in \ell^2(\Gamma_{\Mc,H})\Big\}
\]
For a fixed $\psi \in  L^2(\RR^d)$ non necessarily in $\Ac_\varphi$ we consider the average samples of any $f\in \Ac_\varphi$
\[
\Lc_\psi f(n,A)=\big\langle f, U(n, A) \psi \big\rangle_{L^2(\RR^d)}=\big\langle f, \psi [A^\top(\cdot-n)]\big\rangle_{L^2(\RR^d)}\,,\quad (n,A)\in \Gamma_{\Mc,H}\,.
\]
Under the hypotheses in Theorem \ref{sampriesz}, there exists a function $\Phi_\psi \in \Ac_\varphi$ such that the sequence $\big\{\Phi_\psi [A^\top(t-n)]\big\}_{(n,A)\in \Gamma_{\Mc,H}}$ is a Riesz basis for $\Ac_\varphi$, and for each $f\in \Ac_\varphi$ we have the sampling expansion
\begin{equation}
\label{rcristal}
f(t)=\sum_{(n,A)\in \Gamma_{\Mc,H}} \big\langle f, \psi [A^\top(\cdot-n)]\big\rangle_{L^2(\RR^d)}\,\Phi_\psi [A^\top(t-n)] \quad \text{in $L^2(\RR^d)$}\,.
\end{equation}
If the generator $\varphi \in C(\RR^d)$ and the function $t\mapsto \sum_{(n,A)}|\varphi [A^\top(t-n)]|^2$ is bounded on $\RR^d$, a standard argument shows that $\Ac_\varphi$ is a reproducing kernel Hilbert space (RKHS) of continuous bounded functions in $L^2(\RR^d)$. As a consequence, convergence in $L^2(\RR^d)$-norm implies pointwise convergence which is absolute and uniform on 
$\RR^d$.

\medskip
\subsection{The pointwise samples case}
Let $\{U(\gamma)\}_{\gamma \in \Gamma}$ be a unitary representation of the group $\Gamma=N\rtimes_\sigma H$ on the Hilbert space $\Hc=L^2(\RR^d)$. If the generator $\varphi\in L^2(\RR^d)$ satisfies that, for each $\gamma \in \Gamma$, the function $U(\gamma)\varphi$ is continuous on $\RR^d$, and the condition
\begin{equation}
\label{bounded}
\sup_{t\in \RR^d} \sum_{\gamma \in \Gamma} \big|[U(\gamma)\varphi](t)\big|^2<+\infty\,,
\end{equation}
then the subspace $\Ac_\varphi$ is a RKHS of continuous bounded functions in $L^2(\RR^d)$. In fact, the following result holds:
\begin{lema}
For any $\{\mathsf{a}(\gamma)\}_{\gammaÊ\in \Gamma} \in \ell^2(\Gamma)$ the series $\sum_\gamma \mathsf{a}(\gamma)\, [U(\gamma)\varphi](t)$ converges pointwise to a  continuous bounded function if and only if for each $\gamma \in \Gamma$, the function $U(\gamma)\varphi$ is continuous on $\RR^d$, and condition \eqref{bounded} holds.
\end{lema}
\begin{proof}
Cauchy-Schwarz inequality and Weierstrass M-test prove the sufficient condition. To prove the necessary condition we follow the arguments in \cite{zhou:99}. Indeed, notice first that choosing the delta sequences in $\ell^2(\Gamma)$ we deduce that each function $U(\gamma)\varphi$ is continuous on $\RR^d$. 

For each fixed $t\in \RR^d$, since the series 
$\sum_{\gammaÊ\in \Gamma} \mathsf{a}(\gamma)\, [U(\gamma)\varphi](t)$ converges for any $\{\mathsf{a}(\gamma)\}_{\gammaÊ\in \Gamma}$ in $\ell^2(\Gamma)$, we obtain that 
$\sum_{\gamma \in \Gamma} \big|[U(\gamma)\varphi](t)\big|^2<+\infty$. Moreover, the functional $\Omega_t: \ell^2(\Gamma) \rightarrow \CC$ defined as 
$\Omega_t \mathsf{a}:=\sum_{\gammaÊ\in \Gamma} \mathsf{a}(\gamma)\, [U(\gamma)\varphi](t)$ is bounded with norm $\|\Omega_t\|^2=\sum_{\gammaÊ\in \Gamma}  \big|[U(\gamma)\varphi](t)\big|^2$ (see, for instance, \cite[p.145]{heuser:82}). Next, for fixed $\mathsf{a}=\{\mathsf{a}(\gamma)\}_{\gammaÊ\in \Gamma} \in \ell^2(\Gamma)$ we consider its associated  function $f_\mathsf{a}(t):=\sum_{\gammaÊ\in \Gamma} \mathsf{a}(\gamma)\, [U(\gamma)\varphi](t)$, $t\in \RR^d$. Since $f_\mathsf{a}$ is bounded on $\RR^d$, we get $\sup_{t\in \RR^d} |\Omega_t \mathsf{a}|=\sup_{t\in \RR^d} |f_\mathsf{a}(t)|<+\infty$. Hence, Banach-Steinhaus theorem concludes that 
\[
\sup_{t\in \RR^d} \|\Omega_t \|=\sup_{t\in \RR^d}\big(\sum_{\gammaÊ\in \Gamma} \big|[U(\gamma)\varphi](t)\big|^2 \big)^{1/2}<+\infty\,.
\]
\end{proof}

Now for a fixed point $p\in \RR^d$ we consider, for each $f\in \Ac_\varphi$, the new samples given by
\begin{equation}
\label{samples2}
\Lc_pf(\gamma):=\big[U(\gamma^{-1})f\big](p)\,,\quad \gamma \in \Gamma\,.
\end{equation}
For each $f=\sum_{\eta\in \Gamma} \mathsf{a}(\eta) U(\eta)\varphi$ in $\Ac_\varphi$ we get
\[
\begin{split}
\Lc_pf(\gamma)&=\Big[\sum_{\eta\in \Gamma} \mathsf{a}(\eta)\,U(\gamma^{-1}\eta)\,\varphi\Big](p)=
\sum_{\eta\in \Gamma} \mathsf{a}(\eta)\big[U(\gamma^{-1}\eta)\,\varphi\big](p)\\
&=\sum_{\eta\in \Gamma} \mathsf{a}(\eta)\overline{\mathsf{f}_p(\gamma^{-1}\eta)}=\big\langle \mathsf{a}, L_\gamma\mathsf{f}_p\big\rangle_{\ell^2(\Gamma)}\,,\quad \gamma \in \Gamma\,,
\end{split}
\]
where  $\mathsf{f}_p(\eta):=\overline{\big[U(\eta)\varphi\big]}(p)$, $\eta \in \Gamma$; notice that $\mathsf{f}_p$ belongs to $\ell^2(\Gamma)$. As a consequence, under the hypotheses in Theorem \ref{sampriesz} on this new $\mathsf{f}_p\in \ell^2(\Gamma)$, a sampling formula as \eqref{sformula} holds for the  data sequence $\big\{ \Lc_p f(\gamma)\big\}_{\gamma \in \Gamma}$.

\medskip

In the particular case of the quasi regular representation of a crystallographic group $\Gamma_{\Mc,H}=\Mc\ZZ^d\rtimes_\sigma H$, for each $f\in \Ac_\varphi$ its samples \eqref{samples2} are the pointwise samples
\[
\Lc_pf(n,A)=\big[U[(n,A)^{-1}]f\big](p)=\big[U(-A^\top n,A^\top)f \big](p)=f(Ap+n)\,, \quad (n,A)\in \Gamma\,.
\]
Thus (under hypotheses in Theorem \ref{sampriesz}), there exists a unique function $\Phi_p\in \Ac_\varphi$ such that for each $f\in \Ac_\varphi$ the sampling formula
\begin{equation}\label{sf}
f(t)=\sum_{(n,A)\in \Gamma} f(Ap+n)\,\Phi_p[A^\top(t-n)]\,,\quad t\in \RR^d
\end{equation}
holds. The convergence of the series in $L^2(\RR^d)$-norm implies pointwise convergence which is absolute and uniform on $\RR^d$. The interpolating function  
$\Phi_p=\Tc_{U,\varphi}(\mathsf{g}_p)$ where $\mathsf{g}_p$ is the generator of the dual Riesz basis (see Theorem \ref{riesz}). Coefficients in the expansion  
$f=\sum_{\gamma\in \Gamma} \mathsf{a}(\gamma) U(\gamma)\varphi$ can be computed from the samples as
\begin{equation}
\label{expa1}
\mathsf{a}(\gamma)=\sum_{\eta\in \Gamma} f(Ap+n)\, L_\eta\mathsf{g}_p(\gamma)
\end{equation}
\subsection{An example involving the infinite dihedral group $D_{\infty}$}

To illustrate the results in this section  we  consider group $\Gamma=D_{\infty}$, a unidimensional crystallographic group, and a real continuous generator 
$\varphi \in L^2(\RR)$ supported in the interval $[0,2]$. Notice that we can check if a system $\{U(\gamma)\varphi(t)\}_{\gamma\in D_{\infty}}=\{\varphi(t-n)\}_{n\in \ZZ}\cup \{ \varphi(n-t)\}_{n\in \ZZ}$ is a Riesz basis for $\mathcal{A}_{\varphi}=\big\{\sum_{n\in \ZZ}a(n)\varphi(t-n)+b(n)\varphi(n-t): a,b\in \ell^2(\ZZ)\big\}$ by using
the Gramian condition (see, for instance, Refs.~\cite{boor:93,jia:91}). For instance, the generator $\varphi(t)=(16 t^2-13)t^2(2-t)^2 \chi_{[0,2]}(t)$,  $t\in \RR$, 
fulfills these requirements. 

\medskip

The aim here is the recovery of any $f\in \mathcal{A}_{\varphi}$ from its samples $\{f(n+p)\}_{n\in \ZZ}\cup \{ f(n-p)\}_{n\in \ZZ}$ with a fixed $p\in (0,1/2)$.  We proceed to check condition $(a)$ in Theorem \ref{sampriesz}. Indeed,  since supp $\varphi\subseteq [0,2]$, we obtain
 $\widehat{\mathsf{f}}_{1}(z)=\varphi(p)+\varphi(p+1)z$ and $\widehat{\mathsf{f}}_{-1}(z)=\varphi(1-p)z^{-1}+\varphi(2-p)z^{-2}$ and then
 (see Eq.\eqref{ej})
\[\det \mathbf{F}(z)=C+D(z+z^{-1})\,,\quad z\in \mathbb{T}\,,
\]
where $  C=\varphi(p)^2+\varphi(p+1)^2-\varphi(1-p)^2-\varphi(2-p)^2$ and $D=\varphi(p)\varphi(p+1)-\varphi(1-p)\varphi(2-p).$ Since
$\det \mathbf{F}(e^{iw})=C+2D\cos(w)$, whenever $|C|>2|D|$ the sampling formula \eqref{sf} holds. It reads \[
f(t)=\sum_{n\in \ZZ} \big\{f(n+p) \Phi_{p}(t-n)+
f(n-p) \Phi_{p}(n-t)\big\}\,,\quad t\in \RR\,, 
\]
where the interpolating function is  $\Phi_{p}(t)=\sum_{n\in\ZZ}\big\{\mathsf{g}_{1}(n)\varphi(t-n)+\mathsf{g}_{-1}(n)\varphi(n-t)\big\}$, \,$t\in \RR$,
with (see Eq.~\eqref{ej1})
\[
\widehat{\mathsf{g}}_{1}(z)=\frac{\varphi(p)+\varphi(p+1)z}{C+D(z+z^{-1})},\quad 
\widehat{\mathsf{g}}_{-1}(z)=-\frac{\varphi(1-p)z^{-1}+\varphi(2-p)z^{-2}}{C+D(z+z^{-1})}\,,\quad z\in \mathbb{T}\,.
\]
Note that, whenever $D=0$, the interpolanting function $\Phi_{p}$ has also compact support. For instance, for $p=1/4$, by choosing the generator
$\varphi(t)=(16 t^2-13)t^2(2-t)^2 \chi_{[0,2]}(t)$,  $t\in \RR$,  we obtain that $D=0$ and $C=\frac{3627}{64}$, and therefore the interpolating function
\[
\Phi_{p}(t)=\text{\scriptsize $\frac{64}{3627}$}\big[\varphi(p)\varphi(t)+\varphi(p+1)\varphi(t+1)-\varphi(1-p)\varphi(1-t)-\varphi(2-p)\varphi(2-t)\big]\,, \quad t\in \RR\,,
\]
has support $[-1,2]$.  Using Eqs.~\eqref{rb}, the computation of coefficients $\mathsf{a}(\gamma)$ in  \eqref{expa1} has condition number 
$(B_{\mathsf{f}}/A_{\mathsf{f}})^{1/2}\approx 4.82$. For this choice of $\varphi$, the $D_{\infty}$-invariant space $\mathcal{A}_{\varphi}$ is a subspace of the space of cardinal splines of degree 6 with nodes at $\ZZ$ and continuous derivative. 
\section{A $C^*$-algebra connection}
\label{section5}
It is known that the Banach space $\ell^1(\Gamma)$ becomes a Banach $\ast$-algebra under convolution but it is not a $C^*$-algebra. To avoid this drawback, it can be used the group $C^*$-algebra of $\Gamma$ denoted by $C^*(\Gamma)$; it is the completion of $\ell^1(\Gamma)$ with respect to the norm
$\|\mathsf{f}\|=\|\Lambda_\mathsf{f}\|_{\mathcal{B}(\ell^2(\Gamma))}$ \cite[II.10.2]{blackadar:06}. 

In Ref.~\cite{taylor:89}  it is proved that the mapping $\mathsf{f}\mapsto \mathbf{F}$ in Definition \ref{defF} is a $C^*$-isomorphism between $C^*(\Gamma)$ and a $C^*$-subalgebra of $\mathcal{M}_{\kappa}\big(C(\widehat{N})\big)$, the $C^*$-algebra of the $\kappa\times \kappa$ matrices with continuous entries on $\widehat{N}$ (see also \cite{rieffel:80}). Thus, this $C^*$-isomorphism provides an explicit description for the group $C^*$-algebra of the semidirect product group $\Gamma$. 
 
Next, we show that for the semidirect product group $\Gamma=N\rtimes_{\sigma} H$, when $N$ is a discrete abelian group and $H$ a finite group, an alternative to  the group $C^*$-algebra of $\Gamma$ is the larger space 
\[
L^*(\Gamma):=\Big\{\, \mathsf{f}\in \ell^2(\Gamma)\,\, :\,\, \widehat{\mathsf{f}}_{h,l}\in L^\infty(\widehat{N}),\,\, h,l\in H \,\Big\}\,.
\]
In Theorem \ref{pe} below we will prove  that $L^*(\Gamma)$ is a $C^*$-algebra, and that the linear map $\Sc$ (see Definition \ref{defF})
\[
L^*(\Gamma) \ni  \mathsf{f} \xmapsto{\,\,\Sc\,\,} \mathbf{F}\in \mathcal{M}_{\kappa}\big(L^\infty(\widehat{N})\big)
\]
defines a  $C^*$-isomorphism, and consequently an isometry, between the $C^*$-algebra $L^*(\Gamma)$ and a $C^*$-subalgebra
of $\mathcal{M}_{\kappa}\big(L^\infty(\widehat{N})\big)$. Thus, the space $L^*(\Gamma)$ allows to consider, in a $C^*$-algebra setting, elements of $\ell^2(\Gamma)$ with discontinuous Fourier transform, such as ideal filters in signal processing applications. In case the group $\Gamma = \ZZ \rtimes_{\sigma} 1_{H}\cong \ZZ$, the space $L^*(\Gamma)$ coincides with the space $A'(\ZZ)$ of pseudomeasures \cite[3.1.8]{benedetto:96}.

 Specifically, by $\mathcal{M}_{\kappa}\big(L^\infty(\widehat{N})\big)$ we denote the  involution algebra  formed by the $\kappa\times \kappa$ matrices with entries in $L^\infty(\widehat{N})$, with pointwise addition and multiplication and where the involution is given by the adjoint matrix.
 Any $\mathbf{A}\in \mathcal{M}_{\kappa}\big(L^\infty(\widehat{N})\big)$ can be represented  by the bounded  operator $\pi_{\mathbf{A}}\in \mathcal{B}\big(L^2_{\kappa}(\widehat{N})\big)$ defined by
$\big(\pi_{\mathbf{A}}\mathbf{x}\big)(\xi):=\mathbf{A}(\xi)\mathbf{x}(\xi)$. With respect to the norm induced by this representation,
\[
\|\mathbf{A}\|_{\mathcal{M}_{\kappa}(L^\infty(\widehat{N}))}:=\|\pi_{\mathbf{A}}\|_{\mathcal{B}(L^2_{\kappa}(\widehat{N}))}=\sup \Big\{ \|\mathbf{A}(\cdot)\, \mathbf{x}(\cdot)\|_{L^2_{\kappa}(\widehat{N})}\, :\, \|\mathbf{x}\|_{L^2_{\kappa}(\widehat{N})}=1\Big\}\,,
\]
the  involution algebra $\mathcal{M}_{n}\big(L^\infty(\widehat{N})\big)$ is a $C^*$-algebra \cite[II.6.6]{blackadar:06}. An estimation for this norm can be found in Ref.\cite{benzi:14}.

\begin{lema}
\label{nn} 
For each $\mathbf{A}\in \mathcal{M}_{\kappa}\big(L^\infty(\widehat{N})\big)$ we have that $\|\mathbf{A}\|_{\mathcal{M}_{\kappa}(L^\infty(\widehat{N}))}= \esup_{\xi\in \widehat{N}}\|\mathbf{A}(\xi)\|_2$, where $\|\mathbf{A}(\xi)\|_2$ denotes the spectral norm of the matrix $\mathbf{A}(\xi)$.
\end{lema}
\begin{proof} 
Since $\displaystyle \|\pi_{\mathbf{A}}\mathbf{x}\|^2_{L^2_{\kappa}(\widehat{N})}=
\int_{\widehat{N}}\mathbf{x}^*(\xi)\mathbf{A}^*(\xi)\mathbf{A}(\xi)\mathbf{x}(\xi) d\mu_{\widehat{N}}(\xi)$, the lemma
can be proved by means of the argument used to prove Theorem \ref{bessel} from equality \eqref{ce}.
\end{proof}

Alternatively, this lemma could be proved by checking that $\mathcal{M}_{\kappa}\big(L^\infty(\widehat{N})\big)$ with the norm $\|\mathbf{A}\|= \esup_{\xi\in \widehat{N}}\|\mathbf{A}(\xi)\|_2$ is a $C^*$-algebra, and having in mind the uniqueness of the  $C^*$-norm.

\begin{teo}
\label{pe} 
The vector space $L^*(\Gamma)$ under the convolution product, the involution defined by $ \mathsf{f}^*(\gamma)= \overline{\mathsf{f}(-\gamma)}$, $\gamma \in \Gamma$, and the norm $\| \mathsf{f}\|_{L^*(\Gamma)}=\|\mathbf{F}\|_{\mathcal{M}_{\kappa}(L^\infty(\widehat{N}))}$ becomes a $C^*$-algebra. Besides, the linear map $\mathcal{S}:\mathsf{f}\mapsto \mathbf{F}$ is a $C^*$-isomorphism  between $L^*(\Gamma)$ and a $C^*$-subalgebra
of $\mathcal{M}_{\kappa}\big(L^\infty(\widehat{N})\big)$. The transform $\mathcal{S}$ changes the order of the multiplication, i.e.,
 $\mathcal{S}( \mathsf{g}\ast  \mathsf{f})=\mathcal{S}(\mathsf{f})\mathcal{S}( \mathsf{g})$, $\mathsf{f}, \mathsf{g}\in L^*(\Gamma)$. 
\end{teo}
\begin{proof} 
We can easily check  that  $\mathcal{S}$ satisfies $\mathcal{S}(\mathsf{f}^*)=(\mathcal{S}\mathsf{f})^*$. According to Theorem \ref{Tdomain} we have that for any $\mathsf{f},\mathsf{g}\in L^*(\Gamma)$, $\mathcal{S}(\mathsf{g}\ast\mathsf{f})(\xi)=
\mathcal{S}\mathsf{f}(\xi)\mathcal{S}\mathsf{g}(\xi)$ a.e. $\xi \in \widehat{N}$. As $\| \mathsf{f}\|_{L^*(\Gamma)}=\|\mathbf{F}\|_{\mathcal{M}_{\kappa}(L^\infty(\widehat{N}))}=\| \mathcal{S}\mathsf{f}\|_{\mathcal{M}_{\kappa}(L^\infty(\widehat{N}))}$, we just need to prove that the range of $\mathcal{S}$ is closed in norm.

In so doing, let us consider
$\mathbf{F}_i= \mathcal{S}\mathsf{f}_i$,  with $\mathsf{f}_i\in L^*(\Gamma)$, and a matrix  
$\mathbf{A}\in \mathcal{M}_{\kappa}\big(L^\infty(\widehat{N})\big)$ such that
$\|\mathbf{F}_i -\mathbf{A}\|_{\mathcal{M}_{\kappa}(L^\infty(\widehat{N}))}\rightarrow 0$ as $i\mapsto\infty$. We have to prove that
$\mathbf{A}$ belongs to the range of $\mathcal{S}$. 

\noindent From Lemma \ref{nn}, $\esup_{\xi \in \widehat{N}} \|\mathbf{F}_i(\xi) -\mathbf{A}(\xi)\|_2\rightarrow 0$ as $i\mapsto\infty$. Having in mind that $\max_{h,l} |b_{h,l}|\le \|\mathbf{B}\|_2$ for any matrix $\mathbf{B}=[\,b_{h,l}\,]$, we obtain that  $\esup_{\xi \in \widehat{N}} |(\mathbf{F}_{n})_{h,l}(\xi) -\mathbf{A}_{h,l}(\xi)|\rightarrow 0$ and then  $\|(\mathbf{F}_{i})_{h,l}-\mathbf{A}_{h,l}\|_{L^\infty(\widehat{N})}\rightarrow 0$ as $i\mapsto\infty$.
Having in mind that $\widehat{N}$ is compact, we also have that  
\begin{equation}
\label{a}
\|(\mathbf{F}_{i})_{h,l}-\mathbf{A}_{h,l}\|_{L^2(\widehat{N})}\rightarrow 0\quad \text{as}\,\,\, i\mapsto \infty,\quad l,h\in H.
\end{equation}
On the other hand, since, for each $h\in H$,  $\mathbf{A}_{h,1_H}\in L^\infty(\widehat{N}) \subset L^2(\widehat{N})$, there exists a unique 
$\mathsf{f}\in \ell^2(\Gamma)$
such that the Fourier transform of  $\mathsf{f}(\cdot,h)$ is $\mathbf{A}_{h,1_H}$.

For any $h\in H$, the sequence $\mathsf{f}_i(\cdot,h)$ converges in $\ell^2(N)$ to $\mathsf{f}(\cdot,h)$ since, by \eqref{a}, its Fourier transform $(\mathbf{F}_i)_{h,1}$ converges in $L^2(\widehat{N})$ to  $\mathbf{A}_{h,1_H}$ the Fourier transform  of $\mathsf{f}(\cdot,h)$. Hence, for any $h,l\in H$, the sequence
  $(\mathsf{f}_i)_{h,l}=\mathsf{f}_i(-\sigma_l(\cdot),l^{-1}h)$ converges in $\ell^2(N)$ to  $\mathsf{f}(-\sigma_l(\cdot),l^{-1}h)=\mathsf{f}_{h,l}$ and then its Fourier transform $(\mathbf{F}_{i})_{h,l}$ converges in $L^2(\widehat{N})$ to $\widehat{\mathsf{f}}_{h,l}$. Thus, by  using \eqref{a} and the uniqueness of the limit we obtain that 
  $\widehat{\mathsf{f}}_{h,l}=\mathbf{A}_{h,l}$, $h,l\in H$, and then $\mathcal{S}\mathsf{f}=\mathbf{A}$.
 \end{proof} 

Theorem \ref{pe} gives a simple description of the convolution $C^*$-algebra $L^*(\Gamma)$. For example, from \eqref{ej}, the $C^*$-algebra $L^*(D_\infty)$ for the infinite dihedral group $D_\infty$ is $C^*$-isomorphic to the $C^*$-algebra of matrices of the type 
\[
\mathbf{A}(z)=\begin{bmatrix}f(z)&g(z^{-1})\\g(z)&f(z^{-1})\end{bmatrix}\,,\quad z\in \mathbb{T}\,,\quad \text{with}\quad f,g\in L^\infty(\mathbb{T})\,,
\]
and the norm $\|\mathbf{A}\|_{\mathcal{M}_{2}(L^\infty(\mathbb{T}))}= \esup_{z\in \mathbb{T}}\|\mathbf{A}(z)\|_2$. 

\medskip

Finally, it is worth to mention that it is possible to give an alternative proof of Theorem \ref{riesz} by using Theorem \ref{pe} and $C^*$-algebras techniques. \\

\medskip 
\noindent{\bf A new proof of Theorem \ref{riesz}}:
\begin{proof} 
Let $\mathcal{B}\big(\ell^2(\Gamma)\big)$ be the $C^*$-algebra of bounded linear operators on $\ell^2(\Gamma)$.
The linear map
\[
L^*(\Gamma) \ni  \mathsf{f} \xmapsto {\,\,\Lambda\,\,} \Lambda_{\mathsf{f}}\in \mathcal{B}\big(\ell^2(\Gamma)\big)
\]
defines a  $C^*$-isomorphism between $L^*(\Gamma)$ and a $C^*$-subalgebra 
of $\mathcal{B}\big(\ell^2(\Gamma)\big)$. 
Indeed,
from Theorem \ref{bessel} (see the second remark  in \ref{remark}), we obtain that any
$ \mathsf{f}\in L^*(\Gamma)$ satisfies $ \Lambda_{\mathsf{f}}\in \mathcal{B}\big(\ell^2(\Gamma)\big)$  and  $\|\mathsf{f}\|_{L^*(\Gamma)}=\|\mathbf{F}\|_{\mathcal{M}_{n}(L^\infty(\Gamma)}=
\esup_{\xi\in \widehat{N}} \|\mathbf{F}(\xi)\|_{2}=
B^{1/2}_{\mathsf{f}}=\|\Lambda_{\mathsf{f}}\|_{\mathcal{B}(\ell^2(\Gamma))}$; using \eqref{pro} in Theorem \ref{Tdomain} we obtain that $[\Lambda_{\mathsf{f}\ast\mathsf{g}}] \mathsf{h}= \Lambda_{\mathsf{g}}(\Lambda_{\mathsf{f}} \mathsf{h})$, for all $\mathsf{f},\mathsf{g}\in L^*(\Gamma)$ and $\mathsf{h}\in \ell^2(\Gamma)$; and from \eqref{analisis} we have $\Lambda_{\mathsf{f}^*}=\Lambda^*_{\mathsf{f}}$  for $\mathsf{f}\in L^*(\Gamma)$.
	 
Hence, from Theorem \ref{pe} we deduce that the operator
\[
\Lambda(L^*(\Gamma)) \ni  \Lambda_{\mathsf{f}} \xmapsto{\,\,\mathcal{S}\Lambda^{-1}\,} \mathbf{F} \in \mathcal{S}(L^*(\Gamma))
\]
 is a  $C^*$-isomorphism.

Assume first  (a), that is $\big\{L_{\gamma}\mathsf{f}\big\}_{\gamma\in \Gamma}$ is a Riesz basis for $\ell^2(\Gamma)$. Then, from Theorem \ref{bessel}, the entries of  the matrix-valued function $\mathbf{F}(\xi)$ belong to $L^\infty(\widehat{N})$ and the upper Riesz bound is $B_{f}<\infty$. Besides,
$\Lambda_{\mathsf{f}}\in \mathcal{B}\big(\ell^2(\Gamma)\big)$ and it is invertible. 
Since $\Lambda(L^*(\Gamma))$ is a unital $C^*$-subalgebra of 
$\mathcal{B}\big(\ell^2(\Gamma)\big)$, and 
 $\Lambda_{\mathsf{f}}$ belongs to $\Lambda(L^*(\Gamma))$, its inverse 
 $\Lambda^{-1}_{\mathsf{f}}$ also belongs to $\Lambda(L^*(\Gamma))$ \cite[Proposition 4.8]{takesaki:02}. Then, by applying the $C^*$-isomorphism $\mathcal{S}\Lambda^{-1}$, we obtain that $\mathbf{F}$ is invertible in the $C^*$-subalgebra $\mathcal{S}(L^*(\Gamma))$ and that the lower Riesz bound is \cite[Proposition 3.6.8]{ole:16}
 \[
 \begin{split}
&\| \Lambda^{-1}_{\mathsf{f}}\|^{-2}_{\mathcal{B}(\ell^2(\Gamma))}=
 \|\mathbf{F}^{-1}\|^{-2}_{\mathcal{M}_{\kappa}(L^\infty(\widehat{N}))}=\big[ \esup_{\xi\in \widehat{N}}\|\mathbf{F}^{-1}(\xi)\|^2_2\big]^{-1}\\&=  \big[\esup_{\xi\in \widehat{N}} \lambda^{-1}_{\text{min}} \mathbf{F}^*(\xi)\mathbf{F}(\xi)\big]^{-1}= \einf_{\xi\in \widehat{N}} \lambda_{\text{min}} \mathbf{F}^*(\xi)\mathbf{F}(\xi)= A_{\mathsf{f}}\,.
 \end{split}
 \]
Hence $A_{\mathsf{f}}>0$, and then, having in mind \eqref{det}, we prove condition (b).
 
 \medskip
Assume now (b). Since the entries of $\mathbf{F}(\xi)$ belong to $L^\infty(\widehat{N})$ and $\einf_{\xi\in \widehat{N}} |\det \mathbf{F}(\xi)|>0$, there exists $[\mathbf{F}^*(\xi)]^{-1}$, a.e. $\xi\in \widehat{N}$; besides, $[\mathbf{F}^*(\xi)]^{-1}\in \mathcal{M}_{\kappa}(L^\infty(\widehat{N}))$. Since $\mathcal{S}(L^*(\Gamma))$ is a $C^*$-subalgebra of $\mathcal{M}_{\kappa}(L^\infty(\widehat{N}))$ and 
$\mathbf{F}^*$ belongs to $\mathcal{S}(L^*(\Gamma))$, its
inverse $(\mathbf{F}^*)^{-1}$ also belongs to $\mathcal{S}(L^*(\Gamma))$. Hence there exists a unique $\mathsf{g}\in L^*(\Gamma)$ such that $\Lambda_{\mathsf{g}}=(\mathbf{F}^*)^{-1}$. By means of the $C^*$-isomorphism 
$\mathcal{S}\Lambda^{-1}$, we deduce that $\Lambda_{\mathsf{g}}\Lambda_{\mathsf{f}^*}=Id$, and then
\[
\sum_{\gamma\in \Gamma} \langle \mathsf{a},L_{\gamma}\mathsf{f} \rangle L_{\gamma}\mathsf{g}=\Lambda_{\mathsf{g}}\big(\Lambda_{\mathsf{f}^*}
\mathsf{a}\big)=
\mathsf{a}\,,\quad \mathsf{a}\in\ell^2(\Gamma).\]
From Theorem \ref{bessel}, the systems $\big\{L_{\gamma}\mathsf{f}\big\}_{\gamma\in \Gamma}$ and $\big\{L_{\gamma}\mathsf{g}\big\}_{\gamma\in \Gamma}$ are Bessel sequences.
Hence, from \cite[Theorem 3.6.6]{ole:16}, the system $\big\{L_{\gamma}\mathsf{f}\big\}_{\gamma\in \Gamma}$ is a Riesz basis for $\ell^2(\Gamma)$ with dual Riesz basis $\big\{L_{\gamma}\mathsf{g}\big\}_{\gamma\in \Gamma}$. 

Finally, since $\mathbf{F}^*(\xi)\mathbf{G}(\xi)= \mathbf{F}^*(\xi)[\mathbf{F}^*(\xi)]^{-1}=\mathbf{I}_\kappa$, a.e. $\xi\in \widehat{N}$, having in mind that 
 $\mathcal{T}_{\Gamma}\mathsf{g}(\xi)$ is the first column of the matrix $\mathbf{G}(\xi)$ and Lemma \ref{T}, we deduce that $\mathsf{g}$ is the unique element in $\ell^2(\Gamma)$ satisfying \eqref{dual1}.
\end{proof}

\bigskip

\noindent{\bf Acknowledgments:} 
This work has been supported by the grant MTM2017-84098-P from the Spanish {\em Ministerio de Econom\'{\i}a y Competitividad (MINECO)}.
\vspace*{0.3cm}


\end{document}